\documentclass[a4paper]{article}

\usepackage[T1]{fontenc}
\usepackage[latin1]{inputenc}
\usepackage[english]{babel}
\usepackage{cite}

\usepackage{amsmath}
\usepackage{amssymb}
\usepackage{amsthm}
\usepackage{mathtools}
\usepackage{dsfont}
\usepackage{textcomp}
\usepackage{mathrsfs}

\usepackage{tikz-cd}
\usepackage{graphicx}
\usetikzlibrary{cd,arrows,chains,matrix,positioning,scopes}
\usepackage{adjustbox}
\usepackage{url} %permette di fare collegamenti ipertestuali

\usepackage{pgfplots}
\pgfplotsset{/pgf/number format/use comma,compat=newest}
\usetikzlibrary{decorations.text}

\usepackage{hyperref}
\hypersetup{colorlinks,%
citecolor=black,%
filecolor=black,%
linkcolor=black,%
urlcolor=black}

\newcommand\restr[2]{{% we make the whole thing an ordinary symbol
  \left.\kern-\nulldelimiterspace % automatically resize the bar with \right
  #1 % the function
  \vphantom{\big|} % pretend it's a little taller at normal size
  \right|_{#2} % this is the delimiter
  }}

{\left\lbrace\begin{array}{@{}l@{}}}%
{\end{array}\right.}

\renewcommand{\P}{\mathds{P}}

\renewcommand{\O}{\mathcal{O}}
\newcommand{\F}{\mathcal{F}}
\newcommand{\G}{\mathcal{G}}
\newcommand{\T}{\mathcal{T}}
\renewcommand{\L}{\mathcal{L}}

\DeclarePairedDelimiter{\intinf}{\lfloor}{\rfloor}

\DeclareMathOperator{\jac}{Jac}
\DeclareMathOperator{\Alb}{Alb}
\DeclareMathOperator{\alb}{alb}
\DeclareMathOperator{\Hom}{Hom}
\DeclareMathOperator{\Pic}{Pic}

\DeclareMathOperator{\im}{Im}

\DeclareMathOperator{\spec}{Spec}
\DeclareMathOperator{\specu}{\underline{Spec}}

\DeclareMathOperator{\rk}{rk}

\DeclareMathOperator{\cha}{char}

\numberwithin{equation}{section}

\theoremstyle{remark}
\newtheorem{remark}{Remark}[section]
\newtheorem{example}[remark]{Example}
\theoremstyle{definition}
\newtheorem{definition}[remark]{Definition}
\theoremstyle{plain}
\newtheorem{theorem}[remark]{Theorem}
\newtheorem{lemma}[remark]{Lemma}
\newtheorem{corollary}[remark]{Corollary}
\newtheorem{proposition}[remark]{Proposition}

\title{Surfaces close to the Severi lines in positive characteristic}
\author{Federico Cesare Giorgio Conti}

\begin{document}

\maketitle

\begin{abstract}
 Let $X$ be a surface of general type with maximal Albanese dimension over an algebraically closed field of characteristic greater than two: we prove that if $K_X^2<\frac{9}{2}\chi(\O_X)$, one has $K_X^2\geq 4\chi(\O_X)+4(q-2)$. Moreover we give a complete classification of surfaces for which equality holds for $q(X)\geq 3$: these are surfaces whose canonical model is a double cover of a product elliptic surface branched over an ample divisor with at most negligible singularities which intersects the elliptic fibre twice.
In addition we expose a similar partial result over algebraically closed fields of characteristic two.

We also prove, in the same hypothesis, that a surface $X$ with $K_X^2\neq 4\chi(\O_X)+4(q-2)$ satisfies $K_X^2\geq 4\chi(\O_X)+8(q-2)$ and we give a characterization of surfaces for which the equality holds. These are surfaces whose canonical model  is a double cover of an isotrivial smooth elliptic surface branched over an ample divisor with at most negligible singularities whose intersection with the elliptic fibre is $4$.
\end{abstract}

\tableofcontents

\section{Introduction}
\label{sec_intro}

Let $X$ be a minimal surface of general type with maximal Albanese dimension (recall that a surface has maximal Albanese dimension if its Albanese morphism is generically finite). We denote by $K_X$ the canonical divisor, by $\chi(\O_X)$ the Euler characteristic of the structure sheaf, by $q$ the dimension of its Albanese variety and by $q'=h^1(\O_X)$ the irregularity (recall that over fields of positive characteristic we have $q'\geq q$, while over the complex numbers $q=q'$ holds). 

In this paper we are interested in characterizing  surfaces which lie on or close to the Severi lines, i.e. surfaces for which the quantity
\begin{equation}
\label{int}
K_X^2-4\chi(\O_X)-4(q-2)
\end{equation}
vanishes or is "small" provided that $K_X^2<\frac{9}{2}\chi(\O_X)$ over algebraically closed fields of positive characteristic. This value is strictly related to the so called Severi inequality, which states that a surface of general type with maximal Albanese dimension satisfies
\begin{equation}
K_X^2\geq 4\chi (\O_X)
\label{sev}
\end{equation} 
and was proved over the complex numbers by Pardini in \cite{par_sev} and  over fields of any characteristic by Yuan and Zhang in \cite{yuan} for fields of positive characteristic.
Barja, Pardini and Stoppino have given a characterization of surfaces over the complex numbers for which the inequality \ref{sev} is indeed an equality in \cite{barparsto}, namely these are surfaces whose canonical model is a double cover of its Albanese variety branched over an ample divisor with at most negligible singularities (in particular $q=2$). An analogous characterization is obtained in \cite{gusunzhou} by Gu, Sun and Zhou over fields of any characteristic. There are many generalizations of the Severi inequality; in particular Lu and Zuo have proved in \cite{lu} a similar inequality involving also the irregularity $q=q'$ over the complex numbers:  a surface of general type and maximal Albanese dimension satisfies
\begin{equation}
K_X^2\geq\min\Bigl\{\frac{9}{2}\chi(\O_X),4\chi(\O_X)+4(q-2)\Bigr\}
\label{lusev}
\end{equation}
or, equivalently, if $K_X^2<\frac{9}{2}\chi(\O_X)$ then $K_X^2\geq 4\chi(\O_X)+4(q-2)$.
They also give conditions for a surface to satisfy the equality
\begin{equation}
\frac{9}{2}\chi(\O_X)>K_X^2=4\chi(\O_X)+4(q-2).
\label{uglu}
\end{equation}
The condition $K_X^2<\frac{9}{2}\chi(\O_X)$ is necessary to prove that there exists an involution $i$ for which the Albanese morphism of $X$ is composed with $i$ (cf. \cite{lu} Theorem 3.1) which is central in their argument. 
There is a single step, \cite{lu} Lemma 4.4(2), where the condition $K_X^2<\frac{9}{2}\chi(\O_X)$ is really needed in their proof and it is not enough to require that  the Albanese morphism of $X$ is composed with an involution.

Finally the author have given a complete classification of surfaces lying close to the Severi lines over the complex numbers in \cite{conti}. Namely, it is showed that a surface of general type with maximal Albanese dimension satisfies $K_X^2=4\chi(\O_X)+4(q-2)$ in case $q\geq 3$ and $K_X^2<\frac{9}{2}\chi(\O_X)$ if and only if its canonical model $X$ is a double cover of a product elliptic surface $C\times E$  over a divisor $R$ with at most negligible singularities. Moreover it is proved that, under the same hypotheses, surfaces for which $K_X^2\neq4\chi(\O_X)+4(q-2)$ satisfy $K_X^2\geq 4\chi(\O_X)+ 8(q-2)$. It is also given a characterization of surfaces for which equality holds when $q\geq 3$: these are surfaces whose canonical model $X$ is isomorphic to a double cover of a smooth isotrivial elliptic surface $Y$ over a curve $C$ of genus $q-1$, branched over a divisor $R$ with at worst negligible singularities for which $K_Y.R=8(q-2)$.

In the present article we extend the results of \cite{lu} and \cite{conti} over fields of positive characteristic. 
As we have already mentioned, the starting point for the proof of Severi type inequalities in \cite{lu} over the complex numbers is the following Theorem which we extend over fields of any characteristic.

\begin{theorem}
Let $X$ be a surface of general type with maximal Albanese dimension and suppose that $K_X^2<\frac{9}{2}\chi(\O_X)$. 
Then there exists a morphism of degree two $f\colon X\to Y$ to a normal surface $Y$ such that the Albanese morphism of $X$ factors through $f$, i.e. the following diagram commutes:
\begin{equation}
\begin{tikzcd}
X\arrow{rr}{f}\arrow{dr}{\alb_X} & & Y\arrow{ld}{}\\
& \Alb(X). & 
\end{tikzcd}
\end{equation}
\label{teo_mio1}
\end{theorem}

Actually in \cite{gusunzhou}, this Theorem is tacitly proved over algebraically closed fields of characteristic different from $2$. In order to prove this we improve the results there obtained, proving the following Theorem (which is an improvement of Theorem 5.14 loc. cit.).

\begin{theorem}
Let $X$ be a minimal surface of general type with maximal Albanese dimension.
Then we have
\begin{equation*}
K_X^2\geq\Bigl(4+\min \{c(X),\frac{1}{2}\}\Bigr)\chi(\O_X)
\label{eq_cxl}
\end{equation*}
for a suitable constant $c(X)$.
\label{teo_mio2}
\end{theorem}

We will define precisely $c(X)$ in Section \ref{sec_albanese}: for now it is enough to know that it is a constant that depends an all the possible factorizations $X\to Y\to\Alb(X)$ of the Albanese morphism $\alb_X$ of $X$ where $X\to Y$ is finite of degree $2$.

Once Theorem \ref{teo_mio1} has been proved, we try to extend the result of \cite{lu} and \cite{conti} over fields of positive characteristic: because the proofs used there are strongly based on the theory of double covers, which behaves wildly over fields of characteristic two, we are only able to extend these results verbatim over algebraically closed fields of characteristic greater than $2$.
What we get are the two following Theorems.

\begin{theorem}[cf. \cite{lu} Theorem 1.3 and \cite{conti} Theorem 1.2 for the same result over the complex numbers]
\label{teo_mio3}
Let $X$ be a minimal surface of general type with maximal Albanese dimension over an algebraically closed field of characteristic greater than $2$ and assume that $K_X^2<\frac{9}{2}\chi(\O_X)$.
Then we have that 
\begin{equation*}
K_X^2\geq 4\chi(\O_X)+4(q-2)
\label{eq:luzuoneq3}
\end{equation*}
where $q$ is the dimension of the Albanese variety of $X$.

In particular equality holds, i.e.
\begin{equation*}
K_X^2=4\chi(\O_X)+4(q-2)
\label{eq_teomioa+}
\end{equation*}
if and only if the canonical model of $X$ is isomorphic to a double cover of a product elliptic surface ($q\geq 3$) $Y=C\times E$ where $E$ is an elliptic curve and $C$ is a curve of genus $q-1$, whose branch divisor $R$ has a most negligible singularities and
\begin{equation*}
R\sim_{lin} C_1+C_2+\sum_{i=1}^{2d}E_i
\label{eq_teomioabranch}
\end{equation*}
where  $E_i$ (respectively $C_i$) is a fibre of the first projection (respectively the second projection) of $C\times E$ and $d>7(q-2)$ or the canonical model of $X$ is a double cover of an Abelian surfaces branched over an ample divisor ($q=2$). 
Moreover we have that $\Alb(X)=\Alb(Y)$ and $q(X)=q'(X)$, i.e. the Picard scheme of $X$ is reduced.
In particular, if $q\geq 3$, the Albanese variety of $X$ is not simple.
\end{theorem}

\begin{theorem}[cf. \cite{conti} Theorem 1.1 for the same result over the complex numbers]
\label{teo_mio4}
Let $X$ be a minimal surface of general type with maximal Albanese dimension with $K_X^2<\frac{9}{2}\chi(\O_X)$ over an algebraically closed field of characteristic greater than $2$ and $q$ be the dimension of the Albanese variety $\Alb(X)$ of $X$.
\begin{enumerate}
\item If $K^2_X>4\chi(\O_X)+4(q-2)$, then $K^2_X\geq 4\chi(\O_X)+8(q-2).$
\item If $q=2$ and $K^2_X>4\chi(\O_X)$, then $K^2_X\geq 4\chi(\O_X)+2.$
\item If $q\geq 3$, equality holds, i.e.
\begin{equation*}
\label{eq_8q-2}
K^2_X=4\chi(\O_X)+8(q-2),
\end{equation*}
if and only if the canonical model of $X$ is isomorphic to a double cover of a smooth isotrivial elliptic surface fibration $Y$ over a curve $C$ of genus $q-1$, branched over a divisor $R$ with at most negligible singularities for which $K_Y.R=8(q-2)$. In particular, we have that $\Alb(X)\simeq\Alb(Y)$ and $q(X)=q'(X)$, i.e. the Picard variety of $X$ is reduced, and the Albanese variety of $X$ is not simple.
\end{enumerate}
\end{theorem}

Over algebraically closed fields of characteristic two, the inequality of Theorem \ref{teo_mio3} is still valid, provided that one adds an ad hoc condition that will be discussed in Section \ref{sec_char2}

We would like to stress that all the inequalities in Theorems \ref{teo_mio3} and \ref{teo_mio4} are sharp for every $q$: in section \ref{sec_ex} we give examples for which the equalities hold.

All the materials presented in this article have been written with more details in the PhD Thesis of the author (\cite{contitesi}) and we will refer to it for some technical results whose proof we do not report here.

The article is organized as follows: in Section \ref{sec_pre} we recall some known facts about elliptic surfaces with maximal Albanese dimension and double covers (with a particular attention over fields of characteristic $2$); in Section \ref{sec_ex} we give examples of surfaces of general type with maximal Albanese dimension close to the Severi lines; in Section \ref{sec_nonhyper} we prove a slope-type inequality for non-hyperelliptic surface fibrations; in Section \ref{sec_albanese} we prove Theorems \ref{teo_mio1} and \ref{teo_mio2} using the results of the previous Section; in Section \ref{sec_severi} we prove Theorem \ref{teo_mio3}; in Section \ref{sec_close} we prove Theorem \ref{teo_mio4}; finally, in Section \ref{sec_char2} we give some partial results over fields of characteristic $2$.
Notice that the results of Sections \ref{sec_nonhyper} and \ref{sec_albanese} are independent with respect to the others so that the reader may skip them and go directly to Section \ref{sec_severi} if they takes for granted Theorem \ref{teo_mio1}.

\paragraph{Notation and conventions}
We work over algebraically closed fields of positive characteristic. All varieties are supposed to be projective. Given a surface $X$ we denote by $\Alb(X)$ its Albanese variety, by $\alb_X\colon X\to \Alb(X)$ its Albanese morphism, by $q(X)$ the dimension of $\Alb(X)$ by $q'(X)=h^1(X,\O_X)$ the irregularity and by $e(X)$ the topological Euler characteristic. In this paper $X$ is a surface of general type with maximal Albanese dimension, $C$ is a curve of genus $g(C)>1$, $E$ is an elliptic curve. 

Given the product $C\times E$, we denote by $\pi_C$ and $\pi_E$ the two projections respectively to $C$ and $E$, and by $E_c$ the fibre of $\pi_C$ over $c\in C$ (sometimes $E$ or $E_i$ if it is not necessary to specify the point $c$), respectively $C_e$ the fibre of $\pi_E$ over $e\in E$. Given $L_C\in\Pic(C)$ and $L_E\in\Pic(E)$ we denote by $L_E\boxtimes L_C=\pi_C^*(L_C)\otimes\pi_E^*(L_E)$. By $c_0$ we mean a fixed point of $C$ and we denote by $\Hom_{c_0}(C,E)$ the group of homomorphisms between $C$ and $E$ which send $c_0$ to the origin of $E$ (the group structure is given by the one on $E$). For every $f\in\Hom(C,E)$ we denote by $f+e\in \Hom(C,E)$ the morphism given by $c\mapsto f(c)+e$ and by $\Gamma_f$ its graph as a divisor on $C\times E$.

We use interchangeably the notion of line bundles and Cartier divisors and we use both additive and multiplicative notations.

Given a scheme $X$ defined over an algebraically closed field of characteristic $p$ we denote by $X^{(n)}$ the scheme which is abstractly isomorphic to $X$ and whose structure morphism $X^{(n)}\to \spec(k)$ is the composition of the the structure morphism $X \to \spec(k)$ with the morphism $\spec(k)\to\spec(k)$ defined by $x\mapsto x^{p^n}$.
We denote by $F_k\colon X^{(n)}\to X^{(n-1)}$ the $k$-linear Frobenius morphism and, for a morphism of scheme $f\colon X\to Y$, we denote by $f^{(n)}\colon X^{(n)}\to Y^{(n)}$ the induced morphism.

\paragraph{Acknowledgement} The author would like to thank his advisor Rita Pardini for useful mathematical discussion concerning the topics of the paper. 

\section{Preliminaries}
\label{sec_pre}

In this section we describe the constructions and we expose preliminary results which will be needed in the proofs of the main Theorems.

\subsection{Elliptic surfaces with maximal Albanese dimension}
First we briefly recall the results presented in \cite{conti} Section 2.1 about the Picard group of a product elliptic surface $C\times E$: observe that there it is assumed that $C\times E$ is defined over the complex numbers, but the proof works over any algebraically closed field. Here $C$ is a curve of genus $g>1$ and $E$ an elliptic curve. 
We denote by $\Hom_{c_0}(C,E)$ the group of morphisms between $C$ and $E$ for which the image of $c_0\in C$ is the origin of the elliptic curve (the group structure is given by the one of $E$) and by $i_c\colon E\to C\times E$ the inclusion defined by $e\mapsto (c,e)$.

\begin{proposition}[\cite{conti} Propositions 2.1 and 2.3]
\label{propo_exactpicce}
In the above settings we have the following split exact sequence of groups:
\begin{equation}
0\rightarrow \Pic (C)\times \Pic(E) \xrightarrow{\alpha} \Pic(C\times E) \xrightarrow{\beta} \Hom_{c_0}(C,E) \rightarrow 0,
\label{pics}
\end{equation}
where $\beta$ is defined by $\beta(D)(c)=i_c^*(D)-i_{c_0}^*(D)$ (here we are using the isomorphism $E\simeq \jac(E)$ given by the Abel-Jacobi map) and the section $s$ of $\beta$ is given by
\[s\colon \Hom_{c_0}(C,E) \to  \Pic(C\times E)\quad s(f)=\Gamma_f-C_0-\sum_{c\in f^{-1}(0)} a_cE_c,\]
where $a_c$ is  the multiplicity of $f$ at $c$.

Moreover, suppose there exists a finite Abelian group $G$ acting freely on $C$, $E$ and diagonally on $C\times E$ (i.e. $g\cdot(c,e)=(g\cdot c,g\cdot e)$). Then it is possible to give to $\Pic(C)$, $\Pic(E)$, $\Pic(C\times E)$ and $\Hom_{c_0}(C,E)$ a $G$-module structure such that the short exact sequence \ref{pics} is an exact sequence of $G$-modules.
\end{proposition}

\begin{lemma}[\cite{conti} Lemma 2.5]
\label{lemma_trivialell}
In the same situation as in the second part of Proposition \ref{propo_exactpicce}, the elliptic  $(C\times E)/G\to C/G$  with general fibre $E$ is trivial if and only if there is a line bundle $L$ on $C\times E$ which is fixed by the action of $G$ for which $L.E=1$ where $E$ (by an abuse of notation) is a general fibre of the first projection.
\end{lemma}

Now we collect some known results about elliptic surfaces of maximal Albanese dimension.
Let $f\colon X\to C$ be an elliptic surface fibration.
Because in the present paper we are interested in surfaces of maximal Albanese dimension, we can assume that that all the fibres are smooth (but possibly multiple).
Indeed thanks to the Kodaira's classification of singular fibres of an elliptic surface we know that every fibre of an elliptic surface which is not smooth is a union of rational curves: in particular it is contracted by the Albanese morphism (cf. \cite{barth} Table V.3 and \cite{schutt} Section 4.1).

Consider the first higher pushforward $R^1f_*\O_X=\L\oplus T$ of $\O_X$ where $T$ is the torsion subsheaf of $R^1f_*\O_X$ and $\L=R^1f_*\O_X/T$ is a line bundle: it is known that a point $c\in C$ is in the support of $T$ if and only if $h^0(F_c,\O_{F_c})\geq 2$ where $F_c$ is the fibre of $f$ over $c$, in particular we have that $F_c$ has to be a multiple fibre (cf. \cite{badescu} page Definition 7.14). 
Recall also that such a fibre is called exceptional.

We have the following formula for the canonical bundle of $X$ (cf. ibidem Theorem 7.15):
\begin{equation}
K_X=f^*(\L^{-1}\otimes\omega_C)\otimes\O_X\Bigl(\sum_{i=1}^ma_iF_i\Bigr)
\label{eq_canellsur}
\end{equation}
where
\begin{itemize}
\item $n_iF_i=F_{c_i}$ are all the multiple fibres of $f$ for $i=1,\ldots m$ with $F_i$ reduced;
\item $0\leq a_i<n_i$;
\item $a_i=n_i-1$ if $F_{b_i}$ is not an exceptional fibre.
\end{itemize}
In particular the canonical divisor $K_X$ is contracted by $f$ and satisfies $K_X^2=0$.

We have the following Lemmas

\begin{lemma}[cf. \cite{ueno} Lemma 3.4]
\label{lemma_ueno}
Let $f\colon X\to C$ be an elliptic surface of maximal Albanese dimension, $\alb_X\colon X\to \Alb(X)$ be the Albanese morphism of $X$ and $AJ_C\colon C\to \jac(C)$ be the Abel-Jacobi map of $C$.
Then the kernel of the induced morphism $\alb_f\colon \Alb(X)\to\jac(C)$ is an elliptic curve which is isogenous to any fibre of $f$.
In particular we have $q(X)=g(C)+1$.
\end{lemma}

\begin{lemma} 
\label{lemma_cossec}
Let $f\colon X\to C$ be an elliptic surface of maximal Albanese dimension. 
Then the topological Euler characteristic $e(X)$ and the Euler characteristic $\chi(\O_X)$ of the structure sheaf of $X$ are equal to zero.
\end{lemma}

\begin{proof}
Because the fibres of an elliptic surface of maximal Albanese dimension are all smooth (but possibly multiple), the value of $e(X)$ follows directly from \cite{cossec} Proposition 5.1.6, while $\chi(\O_X)=0$ follows from Noether's formula and $K_X^2=0$.
\end{proof}

\begin{remark}
\label{rem_irr}
Let $f\colon X\to C$ be an elliptic surface of maximal Albanese dimension.
By the Leray spectral sequence 
\begin{equation}
E^{pq}_2=H^p(C,R^qf_*\O_X)\Rightarrow H^{p+q}(X,\O_X)
\label{eq_ssellfib}
\end{equation}
we derive (using the low-degree exact sequence)
\begin{equation}
\begin{split}
0\to H^1(C,\O_C)\to H^1(X,\O_X)\to H^0(C,R^1f_*\O_X)\to\\
\to H^2(C,\O_C)=0\to H^2(X,\O_X)\to H^1(R^1f_*\O_X)\to 0.
\end{split}
\label{eq_lowdegreeterm}
\end{equation}

In particular, by Riemann-Roch, we have
\begin{equation}
\begin{split}
\chi(\O_X)=\chi(&\O_C)-\chi(R^1f_*\O_X)=\\
=-\deg(R^1f_*\O_X)&=-\deg(\L)-h^0(C,T)
\end{split}
\label{eq:eulerellfib}
\end{equation}
and
\begin{equation}
q'(X)=g(C)+h^0(C,\L)+h^0(C,T).
\label{eq_irralb}
\end{equation}

As we have we have shown in Lemmas \ref{lemma_ueno} and \ref{lemma_cossec} we have that $\chi(\O_X)=0$ and $q'(X)\geq q(X)=g(C)+1$.
In particular, if there are no exceptional fibres (e.g. there are no multiple fibres), we  easily derive $R^1f_*\O_X=\O_C$ and $q'(X)=g(C)+1=q(X)$.
Observe also that 
\begin{equation}
-\deg(\L)=h^0(C,T)\geq 0
\label{eq_degl0}
\end{equation}
and that $-\deg(\L)=0$ implies that there are no exceptional fibres.
\end{remark}

We recall the following result about smooth elliptic surface fibration: a detailed proof of this can be found in \cite{contitesi} Corollary 1.7.11.

\begin{proposition}
Let $f\colon X\to C$ be a smooth elliptic surface fibration, i.e. all the fibres of $f$ are smooth curves of genus $1$, with maximal Albanese dimension. 
Suppose that there exists a line bundle $L$ such that $L.F=d>0$ is coprime with the characteristic of the ground field.
Then $f\colon X\to C$ becomes trivial after a suitable \'etale base change.
More precisely, there exists a subgroup $G$ of the $d$-torsion points of the general fibre $F$ acting by translation on $F$, freely on a smooth curve $C'$ and diagonally on $C'\times F$ such that 
\begin{equation}
\begin{tikzcd}
C'\times F\arrow{d}{}\arrow{r}{} & X \arrow{d}{f}\\
C'\arrow{r}{} & C
\end{tikzcd}
\label{eq_isoell+0}
\end{equation}
is commutative with $C'/{G}=C$, $(C'\times F)/{G}=X$ and the horizontal arrows are the quotients by ${G}$.

In particular, if $q(X)=2$, then $X$ is an Abelian variety.
\label{propo_isoell}
\end{proposition}

We recall also the following result on elliptic surfaces.

\begin{lemma}[\cite{Mukai} Proposition 3.3]
\label{lemma_kodvan}
Let $f\colon X\to C$ be an elliptic surface.
Then Kodaira vanishing Theorem holds, i.e. for every ample line bundle $L$ we have that $h^1(X,L^{-1})=0$.
\end{lemma}

\subsection{Double covers}
Now we recall some facts about double covers of surfaces.
If the characteristic of the ground field is different from two, the theory is classical and the reader may refer to \cite{barth}.
Over fields of characteristic equal to $2$ the theory is more complicated and references are \cite{cossec} section 0.1 and \cite{gusunzhou} section 6 in the general setting and \cite{miyapet},  \cite{ekeins}, \cite{ekecan} and \cite{rudakov} for the case of inseparable double cover given by foliations and rational vector fields.
In Chapter 2 of \cite{contitesi} the author has collected all the results about double cover and explained in detail the connections between them.

Let $f\colon X\to Y$ be a flat double cover from a normal surface $X$ to a smooth surface $Y$.
There exists an affine covering $\{V_i\}$ of $Y$, for which $U_i=f^{-1}(V_i)$ is still affine and the induced morphism on regular functions makes $B_i=\O_X(U_i)$ an $A_i=\O_Y(V_i)$-module generated by two elements.

By the miracle flatness Theorem $B=f_*\O_X$ is a locally free $\O_Y$-algebra of rank two and $X\simeq\specu(B)$, where $\specu(B)$ is the relative spectrum of the sheaf of algebras $B$.

Up to a refinement of the affine covering $\{V_i\}$, we have that
\begin{equation}
B_i=A_i[t_i]/(t_i^2+a_it_i+b_i),
\label{eq_doublecov}
\end{equation}
where $a_i, b_i\in A_i$.
Clearly two generators of $B_i$ as an $A_i$-module are $1$ and $t_i$. 
Changing coordinates, the glueing conditions give us that 
\begin{equation}
t_i=g_{ij}t_j+c_{ij},
\label{eq_ti}
\end{equation}
with $g_{ij}\in A_{ij}^*=\O_{Y}(V_{ij})^*$ and $c_{ij}\in A_{ij}=\O_Y(V_{ij})$.
This shows that $f_*\O_X$ fits into the following short exact sequence
\begin{equation}
0\to\O_Y\to f_*\O_X\to L^{-1}\to 0
\label{eq_seqdouble}
\end{equation}
where the transition functions of the sections of $L$ are $g_{ij}$ and the transition functions of the sections $f_*\O_X$ are the matrices
\begin{equation}
\begin{pmatrix}
	1 & -g_{ij}^{-1}c_{ij}\\
	0 & g_{ij}^{-1}
\end{pmatrix}
\label{eq_transitions}
\end{equation}
(indeed if the basis of a free module transforms via a matrix $A$, then the sections expressed in that basis transform via the transpose of the inverse matrix $^tA^{-1}$).
Notice that $c_{ij}$ is the cocycle associated with the short exact sequence \ref{eq_transitions}, i.e. it splits if and only if $c_{ij}$ is cohomologous to zero.
The transition functions of $a_i$ and $b_i$ are (cf. \cite{contitesi} Equations 2.7 and 2.8)
\begin{equation}
\label{eq_ai}
a_i=g_{ij}a_j-2c_{ij}
\end{equation}
and
\begin{equation}
\label{eq_bi}
b_i=g_{ij}^2b_j-a_jg_{ij}c_{ij}+c_{ij}^2.
\end{equation}
\begin{definition}
The line bundle $L$ appearing in the exact sequence \ref{eq_seqdouble} is called the line bundle associated with the double cover $f\colon X\to Y$.
\end{definition}

The formula for the canonical divisor of $X$ is the following (cf. \cite{cossec} Proposition 0.1.3):
\begin{equation}
K_X=f^*(K_Y+L).
\label{eq_candoubcov}
\end{equation}

If the characteristic of the ground field is different from $2$, replacing $t_i$ by $t_i+\frac{1}{2}a_i$, we may assume that $a_i=0$, in particular (cf. Equation \ref{eq_ai}) $c_{ij}=0$ and the sequence \ref{eq_seqdouble} splits.
This means that 
\begin{equation}
f_*\O_X=\O_Y\oplus L^{-1}.
\label{eq_splitdouble}
\end{equation}

Equation \ref{eq_bi} becomes
\begin{equation}
b_i=g_{ij}^2b_j,
\label{eq_bisplit}
\end{equation}
in particular the $b_i$'s glue to a section $b$ of $L^2$ which defines a divisor $R$: the branch divisor of $f$.
In particular we obtain the classical result that, over an algebraically closed field of characteristic different from $2$, a double cover $f\colon X\to Y$ is uniquely determined by $L^{\otimes2}=\O_Y(R)$ where $R$ is the branch divisor of $f$.
It is also clear by the local equation $t_i^2+b_i=0$ that $X$ is normal if and only if the branch divisor $R$ is reduced and is smooth if and only if $R$ is.

If the characteristic of the ground field is $2$, Equations \ref{eq_ai} and equations \ref{eq_bi} become
\begin{equation}
a_i=g_{ij}a_j
\label{eq_ai2}
\end{equation}
and
\begin{equation}
b_i=g_{ij}^2b_j+a_jg_{ij}c_{ij}+c_{ij}^2.
\label{eq_bi2}
\end{equation}

It is then clear that a double cover $f\colon X\to Y$ in characteristic $2$ is uniquely determined by a line bundle $L$, a section $a\in H^0(Y,L)$ (which may be trivial) and local section  $b_i$ which satisfies Equation \ref{eq_bi2}.

\begin{definition}
\label{def_split2}
A double cover in characteristic $2$ is said to be splittable if the exact sequence \ref{eq_seqdouble} splits.
\end{definition}

\begin{proposition}[\cite{cossec} Proposition 0.1.2]
\label{propo_inssep}
A double cover in characteristic $2$ is separable if and only if the section $a$ of $L$ defined by equation \ref{eq_ai2} is not identically zero.
\end{proposition}

\begin{definition}
\label{def_branch2}
If $f\colon X\to Y$ is a separable (splittable or not) double cover, we define the branch divisor $R$ to be the zero divisor of the section $a$ of $L$.
In this case we have $\O_Y(R)=L$.

If $f\colon X\to Y$ is an inseparable splittable double cover, we define the branch divisor $R$ to be the zero divisor of $b$ which is the section given by the glueing of the $b_i$'s. In this case we have
$L^{ 2}=\O_Y(R)$ as when the characteristic is different from two.
\end{definition}

\begin{remark}
Notice that a splittable double cover in characteristic $2$ is uniquely determined by a line bundle $L$ and two effective divisors $R_1$ and $R_2$ such that
\begin{equation}
L=\O_Y(R_1)
\label{eq_ram12}
\end{equation} 
and
\begin{equation}
L^{\otimes2}=\O_Y(R_2)
\label{eq_ram22}
\end{equation}
where $R_1$ coincides with the zero locus of the global section $a$ of $L$ while $R_2$ with the zero locus of the global section $b$ of $L^{\otimes2}$.
Moreover it is inseparable if and only if $R_1=0$.
\label{rem_splitdoublecover2}
\end{remark}

\begin{lemma}[cf. \cite{cossec} Remark 0.2.2]
\label{lemma_branchan}
Let $f\colon X\to Y$ be a separable double cover with branch divisor $R$.
Then the singular points of $X$ lying over smooth points of $R$ (if they exist) are rational singularities of type $A_n$.
In particular if $R$ is smooth, $X$ has at most rational double points as singularities.
\end{lemma}
\begin{proof}
For a detailed proof, see \cite{contitesi} Lemma 2.1.17.
\end{proof}

\begin{remark}
\label{rem_effdouble}
Notice that, except when the double cover is inseparable and non-splittable, we have that $L$ (or a multiple of $L$) is an effective divisor.
The presence of inseparable non-splittable double covers where neither $L$ nor one of its multiple is effective is the main reason why the proof we present here for Severi type inequalities works only over fields of characteristic different from $2$.
\end{remark}

\begin{remark}
\label{alb}
Suppose that we have a separable double cover $f\colon X\to Y$ with non-trivial smooth branch divisor $R$ in any characteristic.
Observe that if $\cha(k)=2$ it can happen that $X$ is singular even if the branch divisor is not (cf. Lemma \ref{lemma_branchan}): in this case by $\Alb(X)$ we mean the Albanese variety of a smooth model of $X$.
Then, if $q(X)=q(Y)$, it follows that $\alb_{f}\colon \Alb(X)\to\Alb(Y)$ is an isomorphism. 

Because the morphism is separable, there exists an involution $\sigma\colon X\to X$ such that $Y=X/\sigma$.
The hypothesis on the ramification tells us that there exists a point $x\in X$ fixed by the action of $\sigma$ which, up to a translation on the Albanese variety, is sent to $0$ by $\alb_X$.
By the universal property of the Albanese variety, we know that $\sigma$ extends to an involution $\alb_\sigma$ on $\Alb(X)$ (which preserves the group structure) such that  the following Cartesian diagram commutes:
\begin{equation}
\begin{tikzcd}
X\arrow{d}{\alb_X}\arrow{r}{\sigma}\arrow[bend left=30]{rr}{f} & X\arrow{d}{\alb_X}\arrow{r}{f} & Y\arrow{d}{\alb_Y}\\
\Alb(X) \arrow{r}{\alb_\sigma}\arrow[bend right=30]{rr}{\alb_f} & \Alb(X)\arrow{r}{\alb_f} & \Alb(Y).
\end{tikzcd}
\label{eq_commalbq=q}
\end{equation}
From $\alb_f\circ\alb_\sigma=\alb_f$ we see that $\alb_f\circ(\alb_\sigma-Id)=0$, from which we derive $\alb_\sigma=Id$ because $\alb_f$ is an isogeny.
The commutative  diagram \ref{eq_commalbq=q}, shows that there exists a morphism $\alpha\colon Y\to \Alb(X)$ such that $\alb_f\circ\alpha=\alb_Y$ from which we can derive $\Alb(X)=\Alb(Y)$.
\end{remark}

It's time to recall double covers associated with 1-foliations and rational vector fields over algebraically closed fields of characteristic $2$ and their relations with double covers presented so far.

\begin{definition}
Let $Y$ be a smoooth surface; a 1-foliationis is a saturated subsheaf $\F$ of the tangent sheaf $\T_{Y}$ which is involutive (i.e. $[\F,\F]\subseteq\F$ where $[\cdot,\cdot]$ denotes the Lie bracket on the tangent space) and 2-closed (i.e. $\F^2\subseteq \F$, which means that the square of any local derivative of $\F$, which is again a local derivative in characteristic two, is again in $\F$). 
It is called smooth if it is a subbundle of $\T_{Y}$. 
\label{def_fol}
\end{definition}

By \cite{hartref} Proposition 1.1 and \cite{hartref} Proposition 1.9, we know that every 1-foliation is a line bundle even if it is not a subbundle of the tangent sheaf $\T_Y$ of $Y$.

Observe that any 1-foliation $\F$ is a subbundle of the tangent sheaf $\T_Y$ (hence it is a smooth foliation) in codimension $1$ and fits in the following short exact sequence
\begin{equation}
0\to \F\to \T_Y\to I_Z L^2\to 0
\label{eq_foliationl}
\end{equation}
where $Z$ is a closed subscheme of codimension $2$, which is called the zero set of the 1-foliation, and $L^2$ has to be intended as a symbol for a line bundle (the exponent $2$ will be clear soon when we will relate the two notions of double covers).

A simple calculation with the Chern classes of these sheaves leads to
\begin{equation}
K_Y+\F=-2L
\label{eq:}
\end{equation}
and
\begin{equation}
\begin{split}
e(Y)=c_2(\T_Y)=&\deg(Z)+2L.\F=\\
=\deg(Z)-\F.(K_Y+\F)&=\deg(Z)-2L.(2L+K_Y).
\label{eq_eulfol}
\end{split}
\end{equation}

\begin{theorem}[cf. \cite{miyapet} Proposition 1.9 and Corollary 1.1 Lecture III or \cite{ekeins} Lemma 4.1 and Corollary 3.4]
\label{teo_doublefol}
There is a bijective correspondence between the set of 1-foliations of rank $1$ on $Y$ and the set of inseparable double covers $g\colon Y\rightarrow X$ where $X$ is normal and its singularities lie in the image of the zero set $Z$ of the foliation.

The correspondence is given as follows: to a double cover $g\colon Y\to X$ we associate the 1-foliation given by the relative tangent sheaf $\T_{Y/X}$, while to a 1-foliation $\F$ we associate the scheme $Y/\F$ which is topologically homeomorphic to $Y$ and whose functions are the functions of $Y$ on which $\F$ vanishes. 
Moreover the k-linear Frobenius morphism $F_k$ factors through $g$, that is the following diagram commutes:
\[
\begin{tikzcd}
Y \arrow{r}{g}\arrow{d}{F_k} & X\arrow{dl}{f}\\
Y^{(-1)}. &
\end{tikzcd}
\]

In addition we have $K_Y=g^*K_X+\F$.
\end{theorem}

Usually it is easier construct a 1-foliation starting with a rational vector field.

\begin{definition}
\label{def_ratvect}
A rational vector field on $Y$ is a rational section $s$ of the tangent space $\T_Y$, that is $s\in H^0(\T_Y\otimes k(Y))$ where $k(Y)$ is the field of rational functions of $Y$.

Two rational vector fields $s_1$ and $s_2$ are said to be equivalent if there exists a rational function $\alpha$ such that $s_1=\alpha s_2$.

A rational vector field is said to be 2-closed if $s^2=\lambda s$ for a constant  $\lambda\in k$ (where the multiplication is the composition of rational derivations): it is called multiplicative if $\lambda=1$, whereas it is called additive if $\lambda=0$.
\end{definition}

Notice that for every rational vector field $s$ we have $[s,s]=0$, in particular, with the notation used for 1-foliations, $s$ is always involutive and that, if $s^2=\lambda s$ with $0\neq\lambda \in k$, then $\sqrt{\lambda^{-1}}s$ is multiplicative.

\begin{proposition}[\cite{contitesi} Proposition 2.2.9]
\label{propo_folratvect}
There exists a bijective correspondence between the set of 1-foliations on a smooth surface $Y$ and the equivalence classes of rational 2-closed vector fields on $Y$.

The correspondence is given as follows: to a rational vector field $s\in H^0(\T_Y\otimes k(Y))$ we associate the 1-foliation $\O_Y(D_2-D_1)\to\T_Y$ where $D_1$ and $D_2$ are the divisorial part of the set of poles and zeroes of $s$; while to a 1-foliation $\F$ we can choose every rational section of $\F$ which, via the inclusion $\F\subseteq\T_Y$, induces a rational vector field.
\end{proposition}

We would like to see how these two theories of inseparable double covers are related to one another and the way to do this is to consider the $k$-linear Frobenius morphism $F_k$.
Indeed it is known that every inseparable double cover between two schemes is a factor of the $k$-linear Frobenius morphism.
That is, let $f\colon X\to Y$ be an inseparable double cover between two surfaces $X$ and $Y$, then there exists a double cover $g\colon Y\to X^{(-1)}$ such that the following diagram commutes
\begin{equation}
\begin{tikzcd}
X\arrow{r}{f}\arrow[bend left=30]{rr}{F_k} & Y \arrow{r}{g} & X^{(-1)}.
\end{tikzcd}
\label{eq_facfrob}
\end{equation}

\begin{definition}
In the situation above, we define $f$ to be the dual inseparable double cover associated with $g$ and vice-versa. 
\end{definition}

Suppose that $X$ is normal and $Y$ is smooth: we have seen that $f$ is uniquely determined by a line bundle $L$ and local sections $b_i$ satisfying relations \ref{eq_bi2} and $g$ is uniquely determined by a 1-foliation $\F$. In \cite{ekecan} pages 105-106 it is proved that the relation between $\F$ and $L$ is
\begin{equation}
-2L=K_Y+\F
\label{eq_lf}
\end{equation}
which clarify the notation of Equation \ref{eq_foliationl}.

If we assume that $X$ is smooth too (and we can always reduce to this case after a finite number of blow-ups) we have the following sequence of morphisms 
\begin{equation}
Y^{(1)}\xrightarrow{g^{(1)}} X\xrightarrow{f} Y\xrightarrow{g} X^{(-1)}
\label{eq_frob1-1}
\end{equation}
where $f\circ g^{(1)}$ and $f\circ g$ are the relative Frobenius morphisms of $Y^{(1)}$ and $X$ respectively.
We can consider $f\colon X\to Y$ as a quotient via a 1-foliation $\G$ and the morphism $g^{(1)}\colon Y^{(1)}\to X$ as a double cover with associated line bundle $M$ . 
A simple calculation leads to the following.
\begin{lemma}
\label{lemma_flgm}
We have $f^* L=\G$, $f^*\F=2M$ and $-2M=K_X+\G$.
\end{lemma} 

\begin{proof}
We know that the canonical bundle of $X$ is isomorphic to $f^*K_Y+\G$ (Theorem \ref{teo_doublefol}) and to $f^*(K_Y+L)$ (Equation \ref{eq_candoubcov}) from which we have $f^* L=\G$.
Then, $-2M=K_X+\G$ is simply Equation \ref{eq_lf} rephrased on $X$.
Moreover, pulling back Equation \ref{eq_lf}, we get $f^*\F=-f^*K_Y-2f^*L=-K_X-\G=2M$.
\end{proof}

Let's now consider a smooth surface $X$ and its Albanese morphism $\alb_X\colon X\to\Alb(X)$. 
In addition, suppose that the dimension of the image of the Albanese morphism is two and that $\alb_X$ is inseparable.

\begin{lemma}
\label{lemma_albinsfact}
There is a canonical way to factor $\alb_X$ as
\begin{equation}
X\xrightarrow{f} Y\xrightarrow{h} \Alb(X),
\label{eq_factalbins}
\end{equation}
where $f$ is inseparable of degree two.
\end{lemma}
\begin{proof}
Because $\alb_X$ is inseparable we have that the image of $\alb_X^*\Omega^1_{\Alb(X)}\to\Omega^1_X$ or, equivalently, the relative tangent sheaf $\T_{X/\Alb(X)}$ has rank $1$ (indeed, because the morphism $\alb_X$ is inseparable, it can not be of maximal rank: cf. \cite{liu} Lemma 6.1.13(b)).
We want to show that $\T_{X/\Alb(X)}$ is a 1-foliation: this will give rise to the desired morphism $f\colon X\to Y$.
Using the short exact sequence of the relative tangent bundle
\begin{equation}
0\to\T_{X/\Alb(X)}\to\T_X\to \alb_X^*\T_{\Alb(X)},
\label{eq_shorttan}
\end{equation}
we have that $\T_{X/\Alb(X)}$ is torsion-free and saturated: in particular it is a line bundle (again by \cite{hartref} Proposition 1.1 and \cite{hartref} Proposition 1.9).
It is also clear that $\T_{X/\Alb(X)}$ is involutive because it has rank one.
The short exact sequence also shows that $\T_{X/\Alb(X)}$ is the sheaf of local sections of $\T_X$ which vanishes on the pull-back of local functions of $\Alb(X)$: this is clearly a condition which is preserved by taking squares, in particular $\T_{X/\Alb(X)}$ is 2-closed and hence a 1-foliation.
It is then clear then that $\alb_X$ factorizes as $X\xrightarrow{f}Y=X/\T_{X/\Alb(X)}\xrightarrow{g}\Alb(X)$.
\end{proof}

\begin{remark}
\label{rem_albinsfact}
In the same setting as above denote by $\widetilde{M}=\im(\alb_X^*\Omega^1_{\Alb(X)}\to\Omega^1_X)$ and by $\overline{M}$ its saturation inside $\Omega^1_X$.
We have the short exact sequence
\begin{equation}
0\to \overline{M}\to \Omega^1_X \to \Omega^1_{X/\Alb(X)}/T(\Omega^1_{X/\Alb(X)})\to 0
\label{eq_shortcantanfol}
\end{equation}
where $T(\Omega^1_{X/\Alb(X)})$ is the torsion part of $\Omega^1_{X/\Alb(X)})$.
Notice that the dual of $\Omega^1_{X/\Alb(X)}/T(\Omega^1_{X/\Alb(X)})$ is equal to the relative tangent sheaf $\T_{X/\Alb(X)}$ (the operation of taking the dual ignores the torsion subsheaf): hence dualizing Equation \ref{eq_shortcantanfol} we get
\begin{equation}
0\to \T_{X/\Alb(X)}\to \T_X \to I_Z\overline{M}^{-1}\to 0,
\label{eq_shortcantanfoldual}
\end{equation}
where $Z$ is the subscheme of zeroes of the foliation $\T_{X/\Alb(X)}$ or, equivalently, the subscheme where $\overline{M}$ is not a subbundle of $\Omega^1_X$.

A Theorem of Igusa \cite{igusa} ensure that the pull-back of global 1-forms via the Albanese morphism is injective, in particular $h^0(\overline{M})\geq q$ where $q$ is the dimension of the Albanese variety.
Let $M=(-K_X-\T_{X/\Alb(X)})/2=-\overline{M}/2$: we have seen that this is the line bundle associated with the double cover $g^{(1)}\colon Y^{(1)}\to X$ (cf. Equation \ref{eq_lf} and Lemma \ref{lemma_albinsfact}) and the square of its inverse has at least $q$ global sections.
This proves that the double cover $g^{(1)}$ is non-splittable (otherwise we would have that $M$ is effective up to a multiple thanks to Remark \ref{rem_effdouble}).
\end{remark}

A classical (at least over fields of characteristic different from $2$) way to solve singularities of a double cover of surfaces $\pi\colon X\to Y$ branched over $R$ is the canonical resolution. We treat separadetely the case of characteristic different from $2$ and the case equal to $2$.

In the first case we have the following diagram:
\begin{equation}
\label{eq_canres}
\begin{tikzcd}
\widetilde{X}=X_t \arrow{d}{f_t}\arrow{r}{\phi_t} & X_{t-1}\arrow{d}{f_{t-1}}\arrow{r}{\phi_{t-1}} & \ldots \arrow{r}{\phi_2} & X_1\arrow{d}{f_1} \arrow{r}{\phi_1} & X \arrow{d}{f}\\
Y_t \arrow{r}{\psi_t} & Y_{t-1}\arrow{r}{\psi_{t-1}} & \ldots \arrow{r}{\psi_2} & Y_1 \arrow{r}{\psi_1} & Y,
\end{tikzcd}
\end{equation}
where the $\psi_i$ are successive blow-ups that resolve the singularities of the branch divisor $R$, the morphism $f_i$ is the double cover branched over $R_i=\psi_i^*R_{i-1}-2m_iE_i$ (equivalently $L_i=\psi_i^*L_{i-1}(-m_iE_i)$), where $E_i$ is the exceptional divisor of $\psi_i$, $m_i=\intinf{d_i/2}$ with $d_i$ the multiplicity in $R_{i-1}$ of the blown-up point and $\intinf{d_i/2}$ denotes the integral part of $d_i/2$. 
One has the following relations (cf. \cite{barth} V.22):
\begin{equation}
K_{\widetilde{X}}^2=2K_Y^2+4K_Y.L+2L^2-2\sum_{i=1}^{t}(m_i-1)^2,
\label{cansq}
\end{equation}
and
\begin{equation}
\chi(\O_{\widetilde{X}})=2\chi(\O_Y)+\frac{1}{2}K_Y.L+\frac{1}{2}L^2-\frac{1}{2}\sum_{i=1}^{t}m_i(m_i-1).
\label{eulchar}
\end{equation}

\begin{definition}
\label{def_cansing}
The singularities of the branch locus $R$ are said to be negligible if $m_i=1$ (or, equivalently, $d_i=2,3$) for all $i=1,\ldots,t$.
\end{definition}

\begin{remark}
\label{rem_cansing}
If the branch divisor $R$ has at most negligible singularities, then  $K_{\widetilde{X}}=(f_t\circ\psi_t\circ\ldots\circ\psi_1)^*(K_Y+L)$ (\cite{barth} Theorem III.7.2).
If we further assume that $\widetilde{X}$ is of general type and that $Y$ contains no rational curves, we have that $\widetilde{X}$ is minimal (in general it can have exceptional divisors even if $Y$ contains no rational curves) and $X$ is its canonical model (ibidem III.7 table 1).
\end{remark}

When the characteristic of the ground field is equal to $2$, we refer to \cite{gusunzhou} Section 6.
The good thing is that we obtain the same diagram  in \ref{eq_canres} as a canonical resolution and the same equations hold for the square of the canonical divisor and the Euler characteristic of the structure sheaf as in equations \ref{cansq} and \ref{eulchar}; the bad thing is that we have not a characterization of the $m_i$'s.

Let $f\colon X\to Y$ be a double cover with $Y$ smooth and $X$ normal, $q\in Y$ be a point such that the (unique) point $p$ lying over it is singular in $X$ and $\psi\colon Y'\to Y$ be the blow-up of $q$.
It is clear that the short exact sequence which determines the double cover $\overline{X}=Y'\times_Y X\to Y'$ is the pull-back via $\psi$ of the short exact sequence \ref{eq_seqdouble} associated with $f$ and the same is true for the local sections defining the double cover, i.e. $a_i'=\psi^* a_i$ and $b_i'=\psi^*b_i$.
This shows that the singularities of $\overline{X}$ (which are defined by equations depending on $a_i'$ and $b_i'$) lie exactly above those of $X$ (which are defined by equations depending on $a_i$ and $b_i$).
In particular we have that $\overline{X}$ has a divisorial singularity and can not be normal.
Let $\nu\colon X' \to\overline{X}$ the normalization  and consider the following diagram:
\begin{equation}
\begin{tikzcd}
X'\arrow[bend left=20]{drrr}{\phi}\arrow[bend right=20]{dddr}{f'}\arrow{dr}{\nu} & & &\\
 & \overline{X} \arrow[ddrr, phantom, "\square"]\arrow{rr}{\widetilde{\phi}}\arrow{dd}{\widetilde{f}} & & X\arrow{dd}{f}\\
 & & & \\
 & Y'\arrow{rr}{\psi} & & Y.
\end{tikzcd}
\label{eq_doublecvonorm}
\end{equation}  
We have the following relation between the two sequences defining the double covers $f'$ and $\widetilde{f}$:
\begin{equation}
\begin{tikzcd}
0 \arrow{r} & \O_{Y'}\arrow{r}{}\arrow[equal]{d}{} & \widetilde{f}_*\O_{\overline{X}}\arrow[hook]{d}{}\arrow{r}{} & (\psi^*L)^{-1}\arrow{r}{}\arrow[hook]{d}{} & 0\\
0 \arrow{r} & \O_{Y'}\arrow{r}{} & f'_*\O_{X'}\arrow{r}{} & (L')^{-1}\arrow{r}{} & 0,
\end{tikzcd}
\label{eq_definingequations}
\end{equation}
where the vertical injections follow from the fact that $X'$ is the normalization of $\overline{X}$.
In particular this proves that $L'=\psi^*L(-mE)$ for a strictly positive integer $m$. 
Reiterating this process we obtain the canonical resolution: the fact that this process leads to a smooth $\widetilde{X}$ is proven in \cite{gusunzhou}   section 6.2 for the inseparable case and section 6.3 for the separable one.
The invariants of $X$ satisfy Equations \ref{cansq} and \ref{eulchar}.
Notice that in this case it is not clear a priori what are the values of the $m_i$'s.

\begin{remark}
\label{rem_mi2}
Observe that, if all the $m_i$'s are equal to $1$ and $\widetilde{X}$ is a surface of general type, all the singularities of $X$ are rational double points (we can derive this from Equation \ref{eulchar}). 
In particular, if $Y$ contains no rational curves, then $\widetilde{X}$ is minimal and $X$ is its canonical model.
\end{remark}

We want to focus a little bit more on the canonical resolution of inseparable double covers.
As we have already seen, for every inseparable double cover $f\colon X\to Y$ there is a canonically defined inseparable double cover $g\colon Y\to X^{(-1)}$.
Fitting this in the canonical resolution we get:
\begin{equation}
\label{eq_canresins}
\begin{tikzcd}
\widetilde{X}=X_t \arrow{d}{f_t}\arrow{r}{\phi_t} & X_{t-1}\arrow{d}{f_{t-1}}\arrow{r}{\phi_{t-1}} & \ldots \arrow{r}{\phi_2} & X_1\arrow{d}{f_1} \arrow{r}{\phi_1} & X \arrow{d}{f}\\
Y_t \arrow{d}{g_t}\arrow{r}{\psi_t} & Y_{t-1}\arrow{r}{\psi_{t-1}}\arrow{d}{g_{t-1}} & \ldots \arrow{r}{\psi_2} & Y_1 \arrow{r}{\psi_1}\arrow{d}{g_1} & Y\arrow{d}{g}\\
X_t^{(-1)} \arrow{r}{\phi_t^{(-1)}} & X_{t-1}^{(-1)}\arrow{r}{\phi_{t-1}^{(-1)}} & \ldots \arrow{r}{\phi_2^{(-1)}} & X_1^{(-1)} \arrow{r}{\phi_1^{(-1)}} & X^{(-1)},
\end{tikzcd}
\end{equation}
where the $\psi_i$ are successive blow-ups among the points of $Y_{i-1}$ dominating singular points of $X_{i-1}^{(-1)}$, the morphism $f_i$ is the double cover defined by the short exact sequence $0\to \O_{Y_i}\to V_i \to L_i^{-1}\to 0$ with $L_i=\psi_i^*L_{i-1}(-m_iE_i)$), where $E_i$ is the exceptional divisor of $\psi_i$ and $\psi_i^*(f_{i-1})_*\O_{X_{i-1}}\subseteq V_i=(f_i)_*\O_{X_i}$; similarly the $g_i$ are defined by the 1-foliation $\F_i=-2L_i-K_{Y_i}$ (cf. Equation \ref{eq_lf}), in particular $\F_i=\psi_i^*\F_{i-1}\bigl((2m_i-1)E\bigr)$. 

Denoting the by $Z_i$ the zero set of the 1-foliation $\F_i$, thanks to Equation \ref{eq_eulfol}, we get that the degree $d_i$ of $Z_i$ is equal to $e(Y_i)+2L_i.(2L_i+K_{Y_i})$.
In particular one has
\begin{equation}
\begin{split}
d_i=e(Y_i)+&2L_i.(2L_i+K_{Y_i})=\\
=e(Y_{i-1}&)+1+\\
+2\psi_i^*L_{i-1}(-m_iE).\Bigl(2\psi_i^*&L_{i-1}(-m_iE)+\psi_i^*K_{Y_{i-1}}(E)\Bigr)=\\
=e(Y_{i-1})+2L_{i-1}.(2L_{i-1}&+K_{Y_{i-1}})-4m_i^2+2m_i+1=\\
=d_{i-1}-4&m_i^2+2m_i+1.
\end{split}
\label{eq_degzi}
\end{equation}
This equation is used in \cite{gusunzhou} to prove that the canonical resolution process terminates after a finite number of blow-ups, and that's because the degree of $Z_i$ has to be non-negative and $X_i$ is smooth as soon as $d_i=0$.

\begin{remark}
Notice that Equation \ref{eq_degzi} implies that if the zero set of the foliation has only points with multiplicity at most $10$, then all the $m_i$'s appearing in the canonical resolution process have to be equal to $1$. 
Indeed if it were not the case assume, for simplicity, $m_1\geq 2$. 
Then we see that  $d_1=\deg(Z)-4m_1^2+2m_1+1\leq \deg(Z)-11$: observe that $Z_1$ coincides with $\psi_1^*Z$ outside of $E_1$ and we get a contradiction.
As far as we know the opposite may not be true: a priori it may happen that $Z$ has points with multiplicity bigger then $10$ and $X$ has only rational singularities.
\label{rem_atmost10}
\end{remark}

We recall the following useful result for the algebraic fundamental group of separable double covers ramified over an ample divisor and inseparable double covers.

\begin{theorem}[cf. \cite{contitesi} Theorem 2.5.3, Corollary 2.5.4 and Remark 2.5.5]
Let $f\colon X\to Y$ be a separable  double cover branched over an ample divisor or an inseparable double cover and let 
\begin{equation}
\begin{tikzcd}
\widetilde{X}\arrow{d}{\widetilde{f}}\arrow{r}{\phi} & X\arrow{d}{f}\\
\widetilde{Y}\arrow{r}{\psi} & Y
\end{tikzcd}
\label{eq_quad}
\end{equation}
be the canonical resolution of $f\colon X\to Y$.
Suppose that all the $m_i$'s appearing are equal to $1$ or, equivalently, that $X$ has only rational double points (e.g. $X$ is the canonical model of $\widetilde{X}$ if $\widetilde{X}$ is a surface of general type).
Then $\phi$, $f$ and, consequently, $f\circ\phi$ induce an isomorphism on the algebraic fundamental groups.
\label{teo_funddoub}
\end{theorem}

\begin{proof}
For a detailed proof of this we refer to \cite{contitesi} Section 2.5
\end{proof}

\section{Examples}
\label{sec_ex}

In this section we give explicit examples of surfaces which satisfy equalities in Theorems \ref{teo_mio3} and \ref{teo_mio4}, proving that all the inequalities are sharp and we also give examples of minimal surfaces of general type with maximal Albanese dimension for which $K_X^2=4\chi(\O_X)+4(q-2)$ holds over an algebraically closed field of characteristic equal to $2$. 
Over algebraically closde fields of characteristic different from $2$, the examples are the same as over the complex numbers.
For this reason we recall them briefly and we refer to \cite{conti} Section 3 for more details.

For the first part of this section we assume that $k$ is an algebraically closed field of characteristic differetn from $2$.
Now we give an an example of a surface satisfying  equality in Theorem \ref{teo_mio3} for $q\geq 3$ (the case $q=2$ is treated in \cite{barparsto} over the complex numbers and in \cite{gusunzhou} over any algebraically closed field).
\begin{example}[double cover of a product elliptic surface, cf. \cite{conti} Example 3.1]
\label{es_1} 
We consider an elliptic surface $Y_0=C\times E$ which is the product of an elliptic curve $E$ and a curve $C$ of genus $g>1$. 
  Let $\pi\colon X\to Y$ be the canonical resolution of the double cover $\pi_0\colon X_0\to Y_0$ associated with $R_0=2L_0$ where $R_0\sim_{hom}2C+2dE$ is a reduced with at most negligible singularities divisor (and it is cearly two-divisible).

Then we have that $X$ is a minimal surface of general type with maximal Albanese dimension with invariants: 
\begin{itemize}
	\item $q'(X)=q(Y_0)=g+1$ (this follows from Lemma \ref{lemma_kodvan} and $(\pi_0)_*\O_{X_0}=\O_{Y_0}\oplus L_0^{-1}$);
	\item $q=q(X)=q(Y_0)=g+1$ (cf. Lemma \ref{lemma_ueno}) in particular the Picard scheme of $X$ is reduced; 
	\item $K_X^2=8(q-2)+4d$ (cf. Equation \ref{cansq});
	\item $\chi(\O_X)=q-2+d$ (cf. Equation \ref{eulchar}).
\end{itemize}
Observe that, if $d>7(q-2)$, we have $K_X^2=4\chi(\O_X)+4(q-2)<\frac{9}{2}\chi(\O_X)$.

Hence we have proved that these surfaces satisfy the conditions of Theorem \ref{teo_mio3}: the next step is to prove that they are the only ones; this will be done in Section \ref{sec_severi}.
\end{example}

Now we give three examples of surfaces for which equality holds in Theorem \ref{teo_mio4}: in the first two cases we have $q(X)\geq 3$, while in the last example $q(X)=2$.

\begin{example}[cf. \cite{conti} Example 3.2]
The easiest possible case is a simple modification of Example \ref{es_1}. We take $Y_0$ as before: in this case we just need to take $R_0\sim_{hom}4C+2dE$ and everything is verified in a completely similar way (as before, we need $d> 7(q-2)$). In this case we have $q'(X)=q(X)=q=g+1, K_X^2=16(q-2)+8d$ and $\chi(\O_X)=2(q-2)+2d$. Hence $K_X^2-4\chi(\O_X)=8(q-2)$.
\label{es_2}
\end{example}

Before the next example, we recall some facts that will be useful. It is known that the Jacobian variety $\jac C'$ of a general curve $C'$ of genus $g'$ is simple (cf. Appendix B of \cite{moonen}). It is also known that, given a general \'etale double cover $C\to C'$ over an algebraically closed field of characteristic different from $2$, its Prym variety $P(C,C')$ is simple (cf. \cite{beauville-prym} or \cite{farkas} Section 3.1). By definition of Prym variety, $\jac C$ is isogenous to $\jac C'\times P(C,C')$. In particular there are no Abelian subvarieties of codimension $1$ of $\jac C$ if $g'>2$ (if $g'=2$, the dimension of $P(C,C')$ is $1$ , in particular $\jac C'$ is an Abelian subvariety of codimension $1$ of $\jac C$). So, for every elliptic curve $E$, the set $\Hom(C,E)$ contains only constant morphisms.

\begin{example}[Double cover of a non-trivial smooth elliptic surface, cf. \cite{conti} Example 3.3]
\label{es_3}
Here we present  an example of surface $X$ of general type satisfying equality in Theorem \ref{teo_mio4}, whose canonical model is a ramified double cover of an elliptic surface which is not a product. We start with $C'$, $C$ and $E$ as above.

Let $G$ be a subgroup of order $2$ of $E$ acting freely on $C$ such that the quotient $C/G$ is $C'$: this action clearly extends diagonally to the product giving a finite morphism of degree two $f\colon C\times E\to Y_0:=(C\times E)/G$. Proposition \ref{propo_exactpicce}, together with the non-existence of surjective morphisms from $C$ to $E$, show that there is no line bundle $L$ fixed by $G$ for which $L.E=1$. By Lemma \ref{lemma_trivialell} this is enough to prove that $Y_0$ is not a product. 

Let $L_C$ and $R_C$ be two line bundles on $C/G$ of degree respectively $d$ and $2d$  such that $2L_C=R_C$. Similarly, let $L_E$ and $R_E$ be two line bundles on $E/G$ of degree respectively $1$ and $2$ such that $2L_E=R_E$. If we take an element $R_0\in |\pi_{C/G}^*R_C+\pi_{E/G}^*R_E|$ with at most negligible singularities  and we denote by $L_0=L_C\boxtimes L_E$, we have $2L_0=\O_{Y_0}(R_0)$; hence we get a double cover $\pi_0\colon X_0\to Y_0$, and after the canonical resolution  we get a minimal smooth surface $X$ of general type with maximal Albanese dimension with invariants:
\begin{itemize}
	\item $q'(X)=q(Y_0)$ (this follows from Lemma \ref{lemma_kodvan} and $(\pi_0)_*\O_{X_0}=\O_{Y_0}\oplus L_0^{-1}$);
	\item $q=q(X)=q(Y_0)$ (cf. Lemma \ref{lemma_ueno}) in particular the Picard scheme of $X$ is reduced; 
	\item $K_X^2=16(q-2)+8d$ (cf. Equation \ref{cansq});
	\item $\chi(\O_X)=2(q-2)+2d$ (cf. Equation \ref{eulchar}).
\end{itemize}

In particular if $d>7(q-2)$ we have that $K_X^2=4\chi(\O_X)+4(q-2)<\frac{9}{2}\chi(\O_X)$.
\end{example}

\begin{example}[Double cover of an Abelian variety ramified over a divisor with a quadruple point, cf. \cite{conti} Example 3.4]
\label{ex_abe}
This is an example of surface of general type with maximal Albanese dimension whose Albanese morphism is composed with an involution with $q(X)=2$ satisfying equality in Theorem \ref{teo_mio4}, i.e.
\[K^2_X=4\chi(\O_X)+2.\]

Let $A$ be an Abelian surface and let $H$ be a very ample divisor. Take general elements $H_1, H_2, H_3, H_4, H_5, H_6, H_7, H_8$  inside the linear system  $|H|$ such that they are smooth, they all pass through a point $p$ and $H_i$ intersects $H_j$ for $i\neq j$ transversely at every intersection point. 
Denote by $H_{i,j}=(H_i\cap H_j)\setminus \{p\}$ for $1\leq i<j \leq 8$: we also require that $H_{i,j}\cap H_{i',j'}=\emptyset$ for every $(i,j)\neq(i',j')$. 
Let $D_1=H_1+H_2+H_3+H_4$ and $D_2=H_5+H_6+H_7+H_8$ and consider the pencil $P=\lambda D_1+\mu D_2$: the base locus of $P$ is $\{p\}\sqcup \bigsqcup_{1\leq i,j\leq 4} H_{i,j+4}$. By a Bertini-type argument, the generic element $R\in P$ is smooth away from $p$ and has a quadruple ordinary point at $p$. 

It is obvious that $R$ is two-divisible, i.e. there exists $L\in\Pic(A)$ with $2L=\O_A(R)$. Consider the double cover $\pi\colon X\to A$ branched over $R$, and take the canonical resolution $\widetilde{\pi}\colon\widetilde{X}\to\widetilde{A}$.
We have that $\widetilde{X}$ is a minimal smooth surface of general type with maximal Albanese dimension with invariants:
\begin{itemize}
	\item $q(\widetilde{X})=q'(\widetilde{X})=2$;
	\item $K_{\widetilde{X}}^2= \frac{1}{2}R^2-2$ (cf. Equation \ref{cansq});
	\item $\chi(\O_{\widetilde{X}})=\frac{1}{8}R^2-1$ (cf. Equation \ref{eulchar}).
\end{itemize}

It is then clear that, if we assume that $H^2>2$, then $K_{\widetilde{X}}^2=4\chi(\O_{\widetilde{X}})+2<\frac{9}{2}\chi(\O_{\widetilde{X}})$.
\end{example}

\begin{remark}
Notice that Example \ref{ex_abe} is not exhaustive among examples of surfaces satisfying the equality in Theorem \ref{teo_mio4} with $q=2$.
Here we sketch briefly how such a classification would go.

Looking carefully at the construction in Example \ref{ex_abe}, we see that the property of the branch divisor $R$ that we have used in the computation of the invariants of $\widetilde{X}$ is the following: the branch divisor $R$ and all the reduced total transforms of $R$ in the canonical resolution process contains exactly one quadruple or quintuple point and all the other singularities are negligible (notice that the definition of this property is similar to the definition of negligible singularities of curves as in \cite{barth} Section II.8).
We consider quadruple or quintuple points because $\intinf{4}=\intinf{5}=2$. Hence, if there is a single point with this property, there is a single $i$ such that $m_i=2$ in the canonical resolution, in particular $K_{\widetilde{X}}^2=2L^2-2$ and $\chi(\O_{\widetilde{X}})=\frac{1}{2}L^2-1$ (cf.  Equations \ref{cansq} and \ref{eulchar}).
For example, a branch divisors containing a single quintuple ordinary point or a triple point with a single tangent is allowed.

Another property we need is that the surface $\widetilde{X}$ coming from the canonical resolution is minimal, otherwise its minimal model $\overline{X}$ would satisfy $K_{\overline{X}}^2-4\chi(O_{\overline{X}})>2$ because $K_{\overline{X}}^2>K_{\widetilde{X}}^2$.
For example, if we take $R$ having a single quintuple point then $\widetilde{X}$ is minimal (hence we obtain another example similar to the one in Example \ref{ex_abe}) while, if it has a triple point with a single tangent, $\widetilde{X}$ contains a $-1$-curve (hence we exclude this case).

One has the to consider case by case (again, in a similar way as in the study of negligible singularities in \cite{barth} Section II.8)  the resolution of all quadruple and quintuple points and see which of them have the properties just reported.
\label{rem_abe}
\end{remark}

From now on we suppose that all the varieties are defined over an algebraically closed field of characteristic $2$.
We give examples of minimal surfaces of general type with maximal Albanese dimension whose ivariants satisfy $K_X^2=4\chi(\O_X)+4(q-2)$ with $q=q(X)=q'(X)$.

First we show examples where the Albanese morphism is a splittable separable degree-two morphism.

\begin{example}[Splittable separable double cover of an Abelian surface]
\label{es_splitsepdoubleabe}
Let $Y_0$ be an Abelian surface and $L_0$ be an ample line bundle on $Y_0$.
Let $R_1$ and $R_2$ be two effective divisor such that $\O_{Y_0}(R_1)=L_0$ and $\O_{Y_0}(R_2)=L_0^{\otimes2}$ and every singular point of $R_1$ is a smooth point of $R_2$ and no irreducible component of $R_1$ is contained twice in $R_2$.
Notice that it is not difficult to find such divisors in an Abelian surface: if we further assume $L_0$ to be very ample, we can even take $R_1$ to be smooth by Bertini's Theorem and $R_2$ whatever effective divisor in the linear series of $2R_1$ different from $2R_1$ itself.

Consider the splittable separable double cover $\pi_0\colon X_0\to Y_0$ associated with $R_1$ and $R_2$ (cf. Remark \ref{rem_splitdoublecover2}): we know that it has only rational double points as singularities (outside the singular points of $R$ we derive this from Lemma \ref{lemma_branchan} while restricting to a neighbourhood of a singular point $p$ of $R$ we have a local equation $t_i^2+a_it_i+b_i$ where, by hypothesis, $a_i\in m_p^2$ and $b_i\in m_p\setminus m_p^2$: hence the Jacobian criterion applies and shows that $X$ is smooth above $p$) and, after the canonical resolution, we get
\begin{equation}
\begin{tikzcd}
X\arrow{d}{\pi}\arrow{r}{\phi} & X_0\arrow{d}{\pi_0}\\
Y\arrow{r}{\psi} & Y_0
\end{tikzcd}
\label{eq_quad}
\end{equation}
where $X_0$ is the canonical model of $X$ (cf. Remark \ref{rem_mi2}).
By the ampleness of $L_0$, it is clear that $X$ is a surface of general type (indeed $K_{X_0}=\pi_0^*L_0$) and that 
\begin{equation}
q'(X)=h^1(Y_0,\O_{Y_0})+h^1(Y_0,L_0^{-1})=2=q(X)
\label{eq_stellaq}
\end{equation}
(cf. \cite{mumford} The Vanishing Theorem Section III.16 for the vanishing of $h^1(Y_0,L_0^{-1})$). 
Moreover the Albanese morphism of $X$ is $\psi\circ\pi$ because otherwise $X$ would be birational to an Abelian surface.

Using Equations \ref{cansq} and \ref{eulchar} we get
\begin{equation}
K_X^2=2L_0^2
\label{eq_mah11bis}
\end{equation}
and
\begin{equation}
\chi(\O_X)=\frac{1}{2}L_0^2
\label{eq_max12bis}
\end{equation}
from which
\begin{equation}
K_X^2-4\chi(\O_X)=0
\label{eq_char2elellbis}
\end{equation}
follows.
\end{example}

\begin{example}[Splittable separable double cover of a product elliptic surface]
\label{es_splitsepdoubleprodell}
Let $C$ be a curve of genus $g\geq 2$, $E$ be an elliptic curve and denote by $Y_0=C\times E$.
Let $R_1$ and $R_2$ be two effective divisors such that $\O_{Y_0}(R_1)=L_0$ and $\O_{Y_0}(R_2)=L_0^{\otimes2}$ such that every singular point of $R_1$ is a smooth point of $R_2$ and no irreducible component of $R_1$ is contained twice in $R_2$ where $L_0\sim_{hom}\O_{Y_0}(C+dE)$ with $d>7(q-2)$ and $q=g+1$.
For example we can take $R_1=C_1+\sum_{i=1}^dE_i$ and $R_2=C_1+C_2+\sum_{i=d+1}^{3d}E_i$ where $C_i$ (respectively $E_i$) are fibres of the natural projection $\pi_E\colon Y_0\to E$ (respectively $\pi_C\colon Y_0\to C$) and $C_1\neq C_2$ (respectively $E_i\neq E_j$ for $i\neq j$).
Notice that, by definition, $C_1$ is the only curve contained in both $R_1$ and $R_2$.

Consider the splittable separable double cover $\pi_0\colon X_0\to Y_0$ associated with $R_1$ and $R_2$ (cf. Remark \ref{rem_splitdoublecover2}): we know, arguing as in Example \ref{es_splitsepdoubleabe}, that it has only rational double points as singularities and after the canonical resolution we get
\begin{equation}
\begin{tikzcd}
X\arrow{d}{\pi}\arrow{r}{\phi} & X_0\arrow{d}{\pi_0}\\
Y\arrow{r}{\psi} & Y_0=C\times E
\end{tikzcd}
\label{eq_quad}
\end{equation}
where $X_0$ is the canonical model of $X$ (cf. Remark \ref{rem_mi2}).
By the ampleness of $L_0$, it is clear that $X$ is a surface of general type (indeed $K_{X_0}=\pi_0^*(L_0+K_{Y_0})$) and that 
\begin{equation}
q'(X)=h^1(Y_0,\O_{Y_0})+h^1(Y_0,L_0^{-1})=q=q(X).
\label{eq_stellaq}
\end{equation}

By Equations \ref{cansq} and \ref{eulchar} we get
\begin{equation}
K_X^2=2K_{Y_0}^2+4K_{Y_0}.L_0+2L_0^2=8(q-2)+4d
\label{eq_mambo1}
\end{equation}
and 
\begin{equation}
\chi(\O_X)=\chi(\O_{Y_0})+\frac{1}{2}K_{Y_0}.L_0+\frac{1}{2}L_0^2=q-2+d
\label{eq_mambo2}
\end{equation}
from which we easily derive
\begin{equation}
K_X^2-4\chi(\O_X)=4(q-2)
\label{eq_mambo3}
\end{equation}
and
\begin{equation}
K_X^2-\frac{9}{2}\chi(\O_X)=\frac{7}{2}(q-2)-\frac{1}{2}d<0
\label{eq_mambo4}
\end{equation}
where the last inequality follows from the assumption $d>7(q-2)$.
\end{example}

Now we recall and deepen a general construction of inseparable double covers given by rational vector fields on a product of two curves taken from \cite{liedtkeuni}.
First of all we recall some notable examples of rational vector fields on curves: all of them are essentially taken from Section 2 of \cite{liedtkeuni} with some variations.

\begin{example}[Multiplicative and additive rational vector fields on an elliptic curve with only two zeroes and poles]
\label{es_multaddvf1}
Over a field of characteristic $2$ an elliptic curve $E$ can be seen as the projectivization inside $\P^2$ of 
\begin{equation}
y^2+\alpha xy+y+x^3
\label{eq_deuring}
\end{equation} 
where $\alpha\in k$ is such that $\alpha^3\neq 1$: this is the so called Deuring normal form (cf. \cite{silverman} Proposition A.1.3).
In particular the $j$-invariant of $E$ is $j(E)=\frac{\alpha^{12}}{\alpha^3-1}$ and $E$ is supersingular (i.e. it has no non-trivial $2$-torsion points) if and only if $\alpha=0$ (ibidem).

Notice that every rational vector field on $E$ is uniquely determined by its value on the function $x$: indeed we can determine its value on $y$ (hence on all the field $k(E)$ of rational functions on $E$) via the condition that it has to vanish on $y^2+\alpha xy+y+x^3$.
Now let $\partial_x$ be the rational vector field on $E$ which correspond to the derivation  by $x$: from the condition $\partial_x(y^2+\alpha xy+y+x^3)=0$ we get
\begin{equation}
(1+\alpha x)\partial_x(y)=x^2+\alpha y:
\label{eq_dery1}
\end{equation}
we claim that $\delta=(1+\alpha x)\partial_x$ is a never-vanishing vector field on $E$.

Indeed $\delta(x)=1+\alpha x=0$ implies $x=\frac{1}{\alpha}$ while $\delta(y)=x^2+\alpha y=0$ implies that $y=\frac{1}{\alpha^3}$, but $(\frac{1}{\alpha},\frac{1}{\alpha^3})$ is not a point of $E$: in particular, because the tangent sheaf of an elliptic curve is trivial, we can extend this to a regular never-vanishing vector field on $E$.

If we consider the rational vector field $\delta_E'=x\delta$: by the previous discussion we easily see that it has a single zero at $(0,0)$ and at $(0,1)$ and a double pole at $\infty$.
We notice that $\delta_E'$ is a multiplicative rational vector field, indeed
\begin{equation}
\delta_E'^2(x)=\delta_E'(x+\alpha x^2)=x+\alpha x^2=\delta_E'(x)
\label{eq_mult1x}
\end{equation}
and
\begin{equation}
\delta_E'^2(y)=\delta_E'(x^3+\alpha xy)=\delta_E'(y^2+y)=\delta_E'(y).
\label{eq_mult1y}
\end{equation}

On the other hand, if we assume $\alpha\neq 0$, a similar calculation shows that $\delta_E=(1+\alpha x)\delta$ is an additive rational vector field with a double zero at $(\frac{1}{\alpha},\frac{1}{\sqrt{\alpha^3}})$ and a double pole at $\infty$.
If $\alpha=0$, i.e. the elliptic curve is supersingular, we believe it is not possible to find an additive rational vector field with only two (possibly equal) zeroes.

Let $m(x)=\prod_{i=1}^d(x+\gamma_i)$, $n(x)=\prod_{i=1}^d(x+\epsilon_i)$, $\beta(x)=\frac{m^2(x)}{\alpha n^2(x)}$ and $\overline{\beta}(x)=\frac{m^2(x)+xn^2(x)}{n^2(x)}$: we can always choose $m(x)$ and $n(x)$ such that $\gamma_i\neq\frac{1}{\alpha},0\neq \epsilon_k$ for all $i$ and $k$ and that all the zeroes $\mu_i$ for $i=1,\ldots 2d+1$ of $m^2(x)+xn^2(x)$ are non-multiple and mutually different from $\frac{1}{\alpha}, 0$ and $\epsilon_i$ ($i=1,\ldots, d$).
A similar computation as above, shows that $\overline{\delta}_E=\beta(x)\delta_E$ and $\overline{\delta}_E'=\overline{\beta}(x)\partial_x$ are, respectively, an additive and a multiplicative rational vector field on $E$.
Moreover their corresponding divisors are, respectively,
\begin{equation}
\begin{split}
2P_{\frac{1}{\alpha}}-2P_\infty+2G_1+2&G_1'+\ldots+2G_d+\\
+2G_d'-2E_1-2E_1'-&\ldots-2E_d-2E_d'
\end{split}
\label{eq_divdelta}
\end{equation}
and
\begin{equation}
\begin{split}
P_{\infty}-2P_{\frac{1}{\alpha}}+M_1+M_1'+\ldots +&M_{2d+1}+M_{2d+1}'-\\
-2E_1-2E_1'-\ldots-&2E_d-2E_d'
\end{split}
\label{eq_divdelta'}
\end{equation}
where $G_i, G_i'$, $E_i, E_i'$ and $M_i, M_i'$  are the points where $x+\gamma_i$, $x+\epsilon_i$ and $x+\mu_i$ vanish, $P_{\frac{1}{\alpha}}=(\frac{1}{\alpha},\frac{1}{\sqrt{\alpha^3}})$ and $P_\infty$ is the point at infinity (notice that, if $\alpha=0$, then $P_{\frac{1}{\alpha}}=P_{\infty}$).
\end{example}

\begin{example}[Multiplicative and additive rational vector fields on a hyperelliptic curve]
\label{es_multaddvf>1}
Over a field of characteristic $2$, a smooth hyperelliptic curve $C$ of genus $g$ can be described as the projectivization inside the total space of the line bundle $\O_{\P^1}(g+1)$ of
\begin{equation}
y^2+yf(x)+g(x)
\label{eq_hyper2}
\end{equation}
where $f(x)$ is a polynomial of degree $g+1$ and $g(x)$ is a polynomial of degree $2g+2$, such that every zero of $f(x)$ is a simple zero of $g(x)$ (cf. \cite{bhosle} Section 1).
Moreover the zeroes of $f(x)$ represent the branch divisor of the hyperelliptic morphism $C\to \P^1$.
In order to have $C$ as tame as possible, i.e. $C\to \P^1$ has the maximum number of branch points, we require $f(x)$ to have no multiple zeroes.

Then we can rewrite the equation of $C$ as
\begin{equation}
y^2+y\prod_{i=0}^{g}(x+\alpha_i)+\prod_{i=0}^{g}(x+\alpha_i)\prod_{i=0}^{g}(x+\beta_i)
\label{eq_hyper2n}
\end{equation}
where $\alpha_i\neq \alpha_j$ for $i\neq j$ and $\beta_i\neq \alpha_j$ for all $i,j\in\{0,\ldots,g\}$.
Denote by $P_i$ for $i=0,\ldots,g$ the ramification points of $C$ %(in particular $K_C=2P_1+\ldots+2P_{g-1}$, cf. \cite{hart} Proposition IV.2.3) 
and by $P_\infty$, $P_\infty'$ the two points of $C$ lying over $\infty\in\P^1$ and by $h=\prod_{i=0}^{g}(x+\beta_i)$.

Consider, as in Example \ref{es_multaddvf1}, the rational vector field $\delta=\partial_x$ on $C$: the equation defining $C$ tells us that
\begin{equation}
\partial_x(y)=\frac{f'(x)y+f'(x)h(x)+f(x)h'(x)}{f(x)}.
\label{eq_partialyhyper}
\end{equation}
We then easily derive that the divisor corresponding to $(x^2+\alpha)\delta$ is $2P_\alpha+2P_\alpha'-2P_0-\ldots-2P_g$ where $\alpha\neq\alpha_i$ for all $i=0,\ldots,g$ and $P_\alpha$ and $P_\alpha'$ are the two points of $C$ lying over $\alpha\in\P^1$: indeed $f(x)$ vanishes of order $2$ on every $P_i$, $X^2+\alpha$ vanishes of order two on $P_\alpha$ and $P_\alpha'$  and the degree of the tangent line bundle of $C$ is $-2g+2$.
In particular, because the divisor associated with $X^2+\alpha$ is $2P_\alpha+2P_\alpha'-2P_\infty-2P_\infty'$, we have that the divisor associated with $\delta$ is $2P_\infty+2P_\infty'-2P_0-\ldots-2P_g$. 

Let $m(x)=\prod_{i=1}^d(x+\gamma_i)$, $n(x)=\prod_{i=1}^d(x+\epsilon_i)$, $\beta(x)=\frac{m^2(x)}{n^2(x)}$ and $\overline{\beta}(x)=\frac{m^2(x)+xn^2(x)}{n^2(x)}$: we can always choose $m(x)$ and $n(x)$ such that $\gamma_i\neq\alpha_j\neq \epsilon_k$ for all $i$, $j$ and $k$ and that all the zeroes $\mu_i$ for $i=1,\ldots 2d+1$ of $m^2(x)+xn^2(x)$ are non-multiple and mutually different from $\alpha_i$ ($i=0,\ldots g$) and $\epsilon_i$ ($i=1,\ldots, d$).
A similar computation to the one in Example \ref{es_multaddvf1}, shows that $\delta_C=\beta(x)\delta$ and $\delta_C'=\overline{\beta}(x)\delta$ are, respectively, an additive and a multiplicative rational vector field on $C$.
Moreover their corresponding divisors are, respectively,
\begin{equation}
\begin{split}
2P_\infty+2P_\infty'+2G_1+2&G_1'+\ldots+2G_d+\\
+2G_d'-2E_1-2E_1'-\ldots-2E_d&-2E_d'-2P_0-\ldots-2P_g
\end{split}
\label{eq_divdelta}
\end{equation}
and
\begin{equation}
\begin{split}
P_\infty+P_\infty'+M_1+M_1'+\ldots +&M_{2d+1}+M_{2d+1}'-\\
-2E_1-2E_1'-\ldots-2E_d-2E_d'-&2P_0-\ldots-2P_g
\end{split}
\label{eq_divdelta'}
\end{equation}
where $G_i, G_i'$, $E_i, E_i'$ and $M_i, M_i'$  are the points where $x+\gamma_i$, $x+\epsilon_i$ and $x+\mu_i$ vanish.
\end{example}

\begin{proposition}
\label{propo_liedtkeaxb}
Let $Y=A\times B$ be a product surface where $g(A), g(B)\geq 1$, $\delta_A$ and $\delta_B$ be two rational additive (respectively multiplicative) vector fields on $A$ and $B$ whose zeroes and poles have multiplicity at most $2$ both of which have at least one zero.
Then the rational vector field $\delta_A+\delta_B$ naturally induced on $Y$ is still additive (respectively multiplicative).

If we consider the induced inseparable double cover $g\colon Y\to X^{(-1)}=Y/(\delta_A+\delta_B)$ we have that the corresponding 1-foliation $\F$ is isomorphic, as a line bundle, to $\O_{Y}(-\pi_A^*\delta_{A_\infty}-\pi_B^*\delta_{B_\infty})$ where $\delta_{A_\infty}$ (respectively $\delta_{B_\infty}$) is the divisor of poles of $\delta_A$ (respectively $\delta_B$) and $\pi_A$, $\pi_B$ are the projection to $A$ and $B$.
Moreover the singularities of $X^{(-1)}$ lie in the image of the points $(a,b)$ where $a$ and $b$ are both zeroes or both poles of $\delta_A$ and $\delta_B$.

If we consider the dual inseparable double cover $f\colon X\to Y$ we have that its associated line bundle $L$ is isomorphic to $\frac{-K_Y-\F}{2}$ and is effective, in particular $X$ is the canonical model of a surface of general type with maximal Albanese dimension and $q'(X)=q(X)=g(A)+g(B)$.
\end{proposition}

\begin{remark}
Of course $X$ is singular (by hypothesis the foliation $\delta_A+\delta_B$ has isolated zeroes), hence by $\Alb(X)$ we mean the Albanese variety of a minimal smooth model of $X$ and by $q(X)$ we mean, as usual, the dimension of $\Alb(X)$.
\label{rem_alb2qq'}
\end{remark}

\begin{proof}
The fact that $\delta_A+\delta_B$ is additive (respectively multiplicative) is immediate.

It is clear by definition that $\delta_A+\delta_B$ vanishes on points $(a,b)$ where both $a$ and $b$ are vanishing points of $\delta_A$ and $\delta_B$ respectively and has poles on points $(a,b)$ where $a$ or $b$ is a pole of $\delta_A$ or of $\delta_B$ respectively.
In particular via the identification of Proposition \ref{propo_folratvect}, the formula for $\F$ and for singular points of $X^{(-1)}$ easily follows (cf. also \cite{liedtkeuni} Proposition 4.2)

Thanks to the assumption on zeroes and poles of $\delta_A$ and $\delta_B$, $X$ has only rational double points (cf. Remarks \ref{rem_mi2} and \ref{rem_atmost10}). 
The formula for $L$ is simply an application of Equation \ref{eq_lf} and is effective because $\delta_{A_\infty}-K_A>0$ (respectively $\delta_{B_\infty}-K_B>0$) by hypothesis.  
Hence, via Equation \ref{eq_candoubcov}, we derive that $X$ is the canonical model of a surface of general type.
It is clear by construction and by universal property of the Albanese variety that $\Alb(X)$ lies between $\Alb(Y)$ and $\Alb(Y^{(1)})=\Alb(Y)^{(1)}$ (i.e. there exists a morphism $\Alb(Y)^{(1)}\to\Alb(X)$ and $\Alb(X)\to\Alb(Y)$ such that the composition is the $k$-linear Frobenius) and, similarly, $\Alb(Y)$ lies between $\Alb(X)$ and $\Alb(X)^{(-1)}$.
In particular $X$ has maximal Albanese dimension and $q(X)=g(A)+g(B)$.
Moreover, $q'(X)=g(A)+g(B)$ follows from K\"unneth formula and the long exact sequence in cohomology associated with the short exact sequence \ref{eq_seqdouble} defining the double cover.
\end{proof}

\begin{corollary}
\label{coro_liedtkeaxb}
In the same situation as in Proposition \ref{propo_liedtkeaxb} assume that either the zeroes of $\delta_A$ and $\delta_B$ are simple or they are double.
Let 
\begin{equation}
\begin{tikzcd}
\widetilde{X}\arrow{r}{\phi}\arrow{d}{f_t} & X\arrow{d}{f}\\
Y_t\arrow{r}{\psi} & Y
\end{tikzcd}
\label{eq_canresaxb}
\end{equation}
be the canonical resolution of $f\colon X\to Y$.
Then $m_i=1$ for all $i=1,\ldots,t$ and $f_t$ is an inseparable non-splittable double cover.
\end{corollary}

\begin{proof}
We have that $m_i=1$ for all $i=1,\ldots,t$ because all the singularities of $X$ are rational double points.
Because $f$ is inseparable, if $f_t$ were splittable $2L_t$ would have at least a global section (cf. Remark \ref{rem_effdouble}): we are going to show that this is not the case and, in order to do it, we need to better understand the description of $L$.

Let $\delta_{A_0}-\delta_{A_\infty}$ (respectively $\delta_{B_0}-\delta_{B_\infty}$) be the divisor associated with $\delta_A$ (respectively $\delta_B$) where $\delta_{A_0}$, $\delta_{A_\infty}$, $\delta_{B_0}$ and $\delta_{B_\infty}$ are effective with degree $d_{A_0}$, $d_{A_\infty}$, $d_{B_0}$ and $d_{B_\infty}$ respectively.
In Proposition \ref{propo_liedtkeaxb} we have seen that $\F=-\pi_A^*\delta_{A_\infty}-\pi_B^*\delta_{B_\infty}$ and that $2L=-K_Y-\F$.
By definition we have that $\delta_{A_\infty}-K_A=\delta_{A_0}$ and $\delta_{B_\infty}-K_B=\delta_{B_0}$, in particular
\begin{equation}
2L=\pi_A^*\delta_{A_0}+\pi_B^*\delta_{B_0}
\label{eq_2laxb}
\end{equation}
from which we derive $2L.B=d_{B_0}$ and $2L.A=d_{A_0}$ where $A$ and $B$ denote a fibre of $\pi_B$ and $\pi_A$ respectively.

In the canonical resolution of $f\colon X\to Y$, we blow-up all the points $(a,b)$ where both $a$ and $b$ are zeroes or poles of $\delta_A$ and $\delta_B$ and $L_t$ is defined by 
\begin{equation}
L_t=\psi^*L(-\sum_{i=1}^tE_i)
\label{eq_ltcanres}
\end{equation}
because $m_i=1$ for all $i$, where the $E_i$'s are all the exceptional divisors of $Y_t$.
Notice also that if $a$ and $b$ are simple (respectively double) zeroes or poles of $\delta_A$ and $\delta_B$ then the singularity above is an $A_1$ (respectively $D_4$) singularity (cf. \cite{liedtkeuni} Proposition 3.2): in particular one single blow-up is enough to solve it (respectively, after the first blow-up we need three blow-ups in three different points in the exceptional divisor).

Now assume that all the zeroes of $\delta_A$ and $\delta_B$ are simple: if $2L_t$ were effective we would find an effective divisor $D$ linearly equivalent to $2L$ which passes twice through all the couples $(a,b)$ as above.
We can show that, in this case,
there would exist a zero $a_0$ of $\delta_A$ such that $\pi_A^*a_0$ is not contained in $D$.
Indeed, suppose, by contradiction, that for every zero $a$ of $\delta_A$ the fibre $\pi^*_Aa$ is contained in $D$: then, by the fact that $2L.A=d_{A_0}$, we would get that all the other components of $D$ are fibres of $\pi_B\colon A\times B\to B$.
In particular it would follow from $2L.B=d_{B_0}$, that $D=\pi_A^*\delta_{A_0}+\pi_B^*D_B$ with $\deg(D_B)=d_{B_0}$.
It is then immediate that such a divisor cannot pass twice through every couple $(a,b)$ as above.
Hence there exists an $a_0$ such that $\pi_A^*a_0$ is not contained in $D$.
Recall that $D$ should pass through the different $d_{B_0}$ points $(a_0,b_i)$ twice, where $b_i$ are all the simple zeroes of $\delta_B$: hence $D.\pi_A^*a_0\geq2d_{B_0}$ which a contradiction to the existence of such a $D$.

Now assume that all the zeroes of $\delta_A$ and $\delta_B$ are double. If $2L_t$ were effective we would find an effective divisor $D$ linearly equivalent to $2L$ which passes through each point $(a,b)$, with $a$ and $b$ zeroes of $\delta_A$ and $\delta_B$, six times (twice for each of the three different tangents corresponding to the $D_4$ singularity we are solving).
In particular we would have that $D.\pi_A^*a_0\geq 4\frac{d_{B_0}}{2}=2d_{B_0}$, which is a contradiction (we are using that in this situation the number of zeroes of $\delta_A$ and $\delta_B$ are $\frac{d_{A_0}}{2}$ and $\frac{d_{B_0}}{2}$ respectively and that one of the tangents may be vertical, in which case that component could have zero intersection with the fibre).
\end{proof}

\begin{remark}
Recall that if we have two rational vector fields $\delta_1$ and $\delta_2$ we say that they are equivalent if there exists a rational function $\alpha$ such that $\delta_1=\alpha\delta_2$ (Definition \ref{def_ratvect}). 
Moreover we have shown that two equivalent rational vector fields are associated with the same 1-foliation $\F$ and, consequently, to the same inseparable double cover.
Clearly on a curve, all rational vector fields are equivalent and they are associated with the $k$-linear Frobenius morphism.

On the other hand, given equivalent rational vector fields $\delta_A$ and $\delta_A'$ on $A$ and $\delta_B$ and $\delta_B'$ on $B$, then, in general, $\delta_A+\delta_B$ is not equivalent to $\delta_A'+\delta_B'$ on $A\times B$.
\end{remark}

We are now ready to give explicit construction of minimmal surfaces of general with maximal ALbanese dimension satisfying $K_X^2=4\chi(\O_X)+4(q-2)<\frac{9}{2}\chi(\O_X)$ over an algebraically closed field of characteristic $2$ whose Albanese morphism is a purely inseparable morphism of degree $2$ onto its image.

\begin{example}[Double cover of a product Abelian surface via a rational vector field]

\label{es_addelles}

This is a generalization of Example 7.1 of \cite{gusunzhou} and Example 6.2 of \cite{takeda} (there it is considered only the case of a double cover of a product elliptic surface with supersingular elliptic fibre via a multiplicative rational vector field).
Let $E_1$ and $E_2$ be two elliptic curves and let $\delta_{E_1}'$ and $\delta_{E_2}'$ be the multiplicative rational vector fields defined in Example \ref{es_multaddvf1}.
Consider the quotient $g_0\colon Y_0=E_1\times E_2\to X_0^{(-1)}$ via the rational vector field $\delta_{E_1}'+\delta_{E_2}'$ and its dual morphism $\pi_0\colon X_0\to Y_0$.
After the canonical resolution we get the following diagram 
\begin{equation}
\begin{tikzcd}
X\arrow{d}{\pi}\arrow{r}{\phi} & X_0\arrow{d}{\pi_0}\\
Y\arrow{r}{\psi} & Y_0=E_1\times E_2:
\end{tikzcd}
\label{eq_canres2es}
\end{equation}
by Proposition \ref{propo_liedtkeaxb} and Corollary \ref{coro_liedtkeaxb} we know that $X$ is a surface of general type of maximal Albanese dimension with $q'(X)=q(X)=g(E_1)+g(E_2)=2$ whose canonical model is $X_0$. 
Moreover (ibidem) the $1$-foliation $\F_0$ associated with $\delta_{E_1}'+\delta_{E_2}'$ is linearly equivalent to $-2E_1'-2E_2'$ and the line bundle $L_0$ associated with $\pi_0$ is linearly equivalent to $E_1+E_2$ where $E_2$, $E_2'$ and $E_1$, $E_1'$ are fibres of the first and the second projection of $E_1\times E_2$ and $\pi_0$ is an inseparable double cover.

Using Equations \ref{cansq} and \ref{eulchar} we get
\begin{equation}
K_X^2=2L_0^2=4
\label{eq_mah11}
\end{equation}
and
\begin{equation}
\chi(\O_X)=\frac{1}{2}L_0^2=1
\label{eq_max12}
\end{equation}
from which
\begin{equation}
K_X^2-4\chi(\O_X)=0
\label{eq_char2elell}
\end{equation}
follows.
It is clear that the Albanese morphism of $X$ is $\psi\circ\pi$ because otherwise we would have that $X$ is birational to an Abelian variety.

If we assume that $E_1$ and $E_2$ are not supersingular and we take the two additive rational vector fields $\delta_{E_1}$ and $\delta_{E_2}$ as in Example \ref{es_multaddvf1} we obtain another construction of a surface of general type lying over the Severi line and everything works as above.

With similar calculation as above, we obtain a surface $X$ satisfying $K_X^2=4\chi(\O_X)$ if in the above discussion we substitute the rational vector fields $\delta_{E_1}$, $\delta_{E_2}$, $\delta_{E_1}'$ and $\delta_{E_2}'$ with $\overline{\delta}_{E_1}$, $\overline{\delta}_{E_2}$, $\overline{\delta}_{E_1}'$ and $\overline{\delta}_{E_2}'$.
\end{example}

\begin{example}[Double cover of a product  elliptic surface via a rational vector field]
\label{es_addhyperelles}

Let $C$ be a hyperelliptic curve of genus $g$ and $E$ be an elliptic curve and let $\delta_{E}'$ and $\delta_{C}'$ be the multiplicative rational vector fields defined in Examples \ref{es_multaddvf1} and \ref{es_multaddvf>1}.
Consider the quotient $g_0\colon Y_0=C\times E\to X_0^{(-1)}$ via the rational vector field $\delta_{C}'+\delta_{E}'$ and its dual morphism $\pi_0\colon X_0\to Y_0$.
After the canonical resolution we get the following diagram 
\begin{equation}
\begin{tikzcd}
X\arrow{d}{\pi}\arrow{r}{\phi} & X_0\arrow{d}{\pi_0}\\
Y\arrow{r}{\psi} & Y_0=C\times E:
\end{tikzcd}
\label{eq_canres2es}
\end{equation}
by Proposition \ref{propo_liedtkeaxb} and Corollary \ref{coro_liedtkeaxb} we know that $X$ is a surface of general type of maximal Albanese dimension with $q'(X)=q(X)=g(C)+g(E)=g+1$ (which we denote simply by $q$) whose canonical model is $X_0$.
Moreover (ibidem) the $1$-foliation $\F_0$ associated with $\delta_{C}'+\delta_{E}'$ is linearly equivalent to $-2C'-\sum_{i=1}^{2d+g+1}2E_i'$ and the line bundle $L_0$ associated with $\pi_0$ is linearly equivalent to $C+\sum_{i=1}^{2d+2}E_i$ where $E_i$, $E_i'$ and $C$, $C'$ are fibres of the first and the second projection of $C\times E$, $d$ is as in the definition of the rational vector field $\delta_C'$ and $\pi_0$ is an inseparable double cover.

Using Equations \ref{cansq} and \ref{eulchar} we get
\begin{equation}
K_X^2=2K_{Y_0}^2+4K_{Y_0}.L_0+2L_0^2=8(g-1)+8(d+1)=8(q-2)+8(d+1)
\label{eq_mah21}
\end{equation}
and
\begin{equation}
\chi(\O_X)=2\chi(\O_{Y_0})+\frac{1}{2}K_{Y_0}.L_0+\frac{1}{2}L_0^2=q-2+2(d+1)
\label{eq_mah22}
\end{equation}
from which we derive
\begin{equation}
K_X^2-4\chi(\O_X)=4(q-2)
\label{eq_char2elellbis}
\end{equation}
and
\begin{equation}
K_X^2-\frac{9}{2}\chi(\O_X)=\frac{7}{2}(q-2)-d-1.
\label{eq_mahboh}
\end{equation}
It easily follows that $K_X^2-\frac{9}{2}\chi(\O_X)<0$ if and only if $d>\frac{7}{2}(q-2)-1$. 

If we assume that $E$ is not not supersingular and we take the two additive rational vector fields $\delta_{C}$ and $\delta_{E}$ as in Examples \ref{es_multaddvf1} and \ref{es_multaddvf>1} we obtain another construction of a surface of general type lying over the Severi lines as above.
\end{example}

\section{Non-hyperelliptic surface fibrations}
\label{sec_nonhyper}

In this Section we prove a bound for the self intersection of the relative canonical bundle of a non-hyperelliptic surface.
This is a refinement of Theorem 3.1 of \cite{gusunzhou}, where the improvement is given by the fact that we allow the constant $c$ apperaing in the statement to take value up to $\frac{1}{2}$ instead of $\frac{1}{3}$ and we allow the base of the fibration to be any curve, and the proof is a small modification of their.
We will use this result in the next Section to prove Theorems \ref{teo_mio1} and \ref{teo_mio2}.

Before considering non-hyperlliptic surface fibrations we recall some facts about Harder-Narasimhan filtration of vector bundles over curves and (strongly) semi-stability.

\begin{definition}
\label{def_semi-}
A vector bundle $E$ on a curve $C$ is said to be (strongly) semi-stable if, for every sub-vector bundle $F\subseteq E$, $\mu(F)\leq\mu(E)$ where $\mu(E)$ is the slope of $E$ and is defined as $\frac{\deg(E)}{\rk(E)}$ (the same holds true for $f^*E$ where $f\colon C'\to C$ is a finite morphism of smooth curves).
\end{definition}

\begin{remark}
It is known (\cite{miyaoka} Proposition 3.2) that, given a finite separable morphism $f\colon C'\to C$, a vector bundle $E$ is semi-stable if and only if $f^*E$ is.
The "if" part holds true even for inseparable morphism, while there are many counterexample for the "only if" part (cf. \cite{langeart} Theorem 1).
That is why the notion of strongly semi-stability has been introduced.
It is also known that strongly semi-stability coincide with semi-stability over curves with globally generated tangent bundle (cf. \cite{mehta} Theorem 2.1 and Remark 2.2), e.g. over rational and elliptic curves.
\end{remark}

Recall that for a vector bundle $E$ there exists a unique Harder-Narasimhan filtration
\begin{equation}
0=E_0\subset E_1\subset \ldots \subset E_n=E,
\label{eq_hnfilt}
\end{equation}
such that the subsequent quotients $E_i/E_{i-1}$ are semi-stable vector bundles with slopes $\mu_i=\mu(E_i/E_{i-1})$ satisfying
\begin{equation}
\mu_{max}(E)=\mu_1>\ldots>\mu_n=\mu_{min}(E)
\label{eq_slopehn}
\end{equation}
which can be constructed inductively from the bottom ($E_1$ is the subbundle with lowest slope) or from the top ($E_{n-1}$ is the subbundle for which the quotient has maximal slope).
An easy calculation shows that 
\begin{equation}
\deg(E)=\sum_{i=1}^nr_i(\mu_i-\mu_{i+1}),
\label{lemma_deghn}
\end{equation}
where $r_i=\rk(E_i)$.

\begin{proposition}
\label{coro_belloneq}
Let $E$ be a vector bundle and $0=E_0\subset E_1 \subset\ldots\subset E_n=E$ be its Harder-Narasimhan filtration with slopes $\mu_1>\mu_2>\ldots>\mu_n$ and let $r_i=\rk E_i$.
Let $0=F_0\subset F_1\subset\ldots\subset F_m=E$ be another filtration of $E$ and let $\widetilde{r}_j$ and $\widetilde{\mu}_j$ numbers such that $\mu_{min}(F_j)\geq\widetilde{\mu}_j$ and $\rk(F_j)\geq\widetilde{r}_j$ for all $j\in\{1,\ldots,m\}$ and $\widetilde{r}_m=\rk(F_m)=r_n=\rk(E)$.
Then we have
\begin{equation}
\deg(E)\geq\sum_{i=1}^m\widetilde{r}_j(\widetilde{\mu}_j-\widetilde{\mu}_{j+1})
\label{eq:}
\end{equation} 
\end{proposition}

\begin{proof}
For every $j\in\{1,\ldots,m\}$ there exists a unique $i\in\{1,\ldots,n\}$, call it $i_j$, such that $F_j\subseteq E_{i_j}$ and $F_j\nsubseteq E_{i_j-1}$: in particular we have a non-trivial morphism $F_j\to E_{i_j}/E_{i_j-1}$, which, thanks to Proposition 1.2.7 of \cite{huy}, implies that $\widetilde{\mu}_j\leq\mu_{min}(F_j)\leq \mu_{max}(E_{i_j}/E_{i_j-1})=\mu_{i_j}$.
We also clearly have that $\widetilde{r}_j\leq\rk(F_j)\leq r_{i_j}$.

\begin{figure}[h]%
\centering
\begin{tikzpicture}[xscale=1, yscale=1, every node/.style={scale=1}]
	\begin{axis}[ticks=none, xmin=-5, xmax=5, ymin=0, ymax=7, axis lines=middle, xlabel=$\mu$, ylabel=$r$, enlargelimits]
	
	\draw[white!0] (0,0) -- (0,7);

	\fill [gray!20] (4,0) -- (4,1) -- (2,1) -- (2,0) -- (4,0);
	\fill [gray!20] (3,1) -- (3,3) -- (1,3) -- (1,1) -- (3,1);
	\fill [gray!20] (1,2) -- (1,4) -- (-2,4) -- (-2,2) -- (1,2);
	\fill [gray!20] (-2,3) -- (-2,4) -- (-3,4) -- (-3,3) -- (-2,3);
	\fill [gray!20] (-1,4) -- (-1,6) -- (-4,6) -- (-4,4) -- (-1,4);

	\draw (4,0) -- (4,1);
	
	\draw (4,1) -- (3,1);
	
	\draw (3,1) -- (3,3);
	
	\draw[dashed] (3,3) -- (1,3);
	
	\draw[dashed] (1,3) -- (1,4);
	
	\draw (1,4) -- (-1,4);
	
	\draw (-1,4) -- (-1,6);
	
	\draw (-1,6) -- (-4,6);
	
	\filldraw  (4,1) circle (1pt);
	\filldraw  (3,3) circle (1pt);
	\filldraw  (1,4) circle (1pt);
	\filldraw  (-1,6) circle (1pt);
	
	\node at (-0.9,6.4) {$(\mu_n,r_n)$};
	\node at (4.1,1.4) {$(\mu_1,r_1)$};
	\node at (1.1,4.4) {$(\mu_{n-1},r_{n-1})$};
	\node at (3.1,3.4) {$(\mu_2,r_2)$};

	\draw (2,0) -- (2,1);
	
	\draw (2,1) -- (1,1);
	
	\draw (1,1) -- (1,2);
	
	\draw[dashed] (1,2) -- (-2,2);
	
	\draw[dashed] (-2,2) -- (-2,3);
	
	\draw (-2,3) -- (-3,3);
	
	\draw (-3,3) -- (-3,4);
	
	\draw (-3,4) -- (-4,4);
	
	\draw (-4,4) -- (-4,6);
	
	\filldraw  (2,1) circle (1pt);
	\filldraw  (1,2) circle (1pt);
	\filldraw  (-2,3) circle (1pt);
	\filldraw  (-3,4) circle (1pt);
	\filldraw  (-4,6) circle (1pt);
	
	\node at (1,0.6) {$(\widetilde{\mu}_1,\widetilde{r}_1)$};
	\node at (0,1.6) {$(\widetilde{\mu}_2,\widetilde{r}_2)$};
	\node at (-4,2.6) {$(\widetilde{\mu}_{m-2},\widetilde{r}_{m-2})$};
	\node at (-4,3.6) {$(\widetilde{\mu}_{m-1},\widetilde{r}_{m-1})$};
	\node at (-4,6.4) {$(\widetilde{\mu}_m,\widetilde{r}_m)$};

	\end{axis}
\end{tikzpicture}
\caption{}%
\label{fig_hn}%
\end{figure}

Viewing the points $(\mu_i,r_i)$ and $(\widetilde{\mu}_j,\widetilde{r}_j)$ in a plane as in Figure \ref{fig_hn} we see that 
\begin{equation}
\deg(E)-\sum_{i=1}^m\widetilde{r}_j(\widetilde{\mu}_j-\widetilde{\mu}_{j+1})=\sum_{i=1}^nr_i(\mu_i-\mu_{i+1})-\sum_{i=1}^m\widetilde{r}_j(\widetilde{\mu}_j-\widetilde{\mu}_{j+1})
\end{equation}
is the area of the grey part, hence the Proposition follows.
\end{proof}

In general, in positive characteristic, the tensor product of semi-stable vector bundles is not semi-stable. 
On the other hand the tensor product of strongly semi-stable vector bundle is strongly semi-stable and it is not hard to prove (cf. \cite{contitesi} Corollary 1.2.22) that
\begin{equation}
\label{eq_tensorstrong}
\mu_{min}(E\otimes F)=\mu_{min}(E)+\mu_{min}(F).
\end{equation}

\begin{theorem}
\label{teo_nefantican}
Let $\pi\colon\P(E)\to C$ a projective bundle over a smooth curve $C$. 
Suppose that all the quotients of the Harder-Narasimhan filtration of $E$ are strongly semi-stable. 
Then the rational divisor $\O_{\P(E)}(1)-\mu_{min}(E)F$ is nef\index{divisor!nef}, i.e. it intersects non-negatively every curve of $\P(E)$, where $F$ is a general fibre of $\pi$.
\end{theorem}

\begin{proof}
It is not difficult to see that the proof of Theorem 3.1 part 1)$\Rightarrow$2) adapts smoothly to this case: for a detailed proof see \cite{contitesi} Theorem 1.3.3.
\end{proof}

Consider a non-hyperelliptic surface fibration $f\colon X\to C$ with nef relative canonical bundle and let $g$ be the arithmetic genus of a generic fibre $F$.
Let $E=f_*\omega_{X/C}$,
\begin{equation}
0=E_0\subseteq E_1\subseteq\ldots\subseteq E_{n-1}\subseteq E_n=E
\label{eq_HNcan2}
\end{equation}
be the Harder-Narasimhan filtration of $E$, suppose that all the quotients $E_i/E_{i-1}$ are strongly semi-stable and let $\mu_i$ and $r_i$ be as above.

Consider the generically surjective morphism $\Phi_i\colon f^*E_i\to\omega_{X/C}$: then the image of $\Phi_i$ is $\im(\Phi_i)=I_{Z_i}\otimes_{O_X}\O_X(-Z_i(E_i))\otimes_{\O_X}\O_X(K_{X/C})$ where $I_{Z_i}$ is the ideal sheaf of a  subscheme of codimension $2$ and $Z_i=Z(E_i)$ is an effective divisor.
We recall that, cf. \cite{xiao} Definition page 452 immediately after Lemma 2, $Z_i$ is called the fixed part of $K_{X/C}$ with respect to $E_i$, $M_i=K_{X/C}-Z_i$ is called the moving part of $K_{X/C}$ with respect to $E_i$ and we define $N_i=M_i-\mu_iF$ (notice that this is a rational divisor).
It is clear by definition that, restricting to a general fibre $F$, $\restr{f^*E_i}{F}$ defines a linear subsystem of $|\restr{K_{X/C}}{F}|$ of dimension equal to $r_i-1$, such that $\restr{M_i}{F}$ and $\restr{Z_i}{F}$ are respectively the moving and the fixed part.
Moreover the degree $d_i$ of the linear system $\restr{M_i}{F}$ is the intersection number $M_i.F=N_i.F$. 

\begin{remark}
\label{propo_xiaonef}
Because all the subsequent quotient of the Harder-Narasimhan filtration of $f_*\omega_{X/C}$ are strongly semi-stable, the (rational) line bundle $N_i$ defined above is nef.
Indeed, essentially by definition of $Z_i$, we know that $\Phi_i\colon f^*E_i\to \O_X(K_{X/C}-Z_i)$ is surjective in codimension $1$: hence it corresponds to a rational map $\phi_i\colon X\dashrightarrow \P(E_i)$ such that $\pi_i\circ \phi=f$, where $\pi_i\colon\P(E_i)\to C$ is the natural projection.
Moreover we have that $\O_X(K_{X/C}-Z_i)=\phi_i^*\O_{\P(E_i)}(1)$.
We notice that $N_i=\phi_i^*(\O_{P(E_i)}(1)-\mu_iF_i)$ where $F_i$ is a general fibre of the projection $\pi_i\colon \P(E_i)\to C$.
By Theorem \ref{teo_nefantican} we easily derive the nefness of $N_i$.
\end{remark}

It is clear, by the fact that restriction $\restr{\phi_i}{F}\colon F\to \P^{r_i-1}$ is non degenerate, that, if $r_i>1$ the image of $\phi_i$ is a surface: denote by $X_i$ a minimal smooth model of this surface.
Moreover let $\delta_i$ be the degree of $\phi_i\colon X\dashrightarrow X_i$ and let $g_i$ be the arithmetic genus of the generic fibre of the surface fibration $X_i\to C$ induced by the projection $\pi_i\colon \P(E_i)\to C$.
It is immediate by construction that $\phi_i$ factors through $\phi_{i+1}$, in particular $\delta_{i+1}|\delta_i$.
If $r_1=1$ (notice that $r_i>r_{i-1}\geq 1$, so that only $r_1$ can be equal to one), then it is clear that $d_1=0$ (indeed in this case we have $\P(E_1)=C$ and $N_1$ is the pull-back of a line bundle on $C$) and we define $\delta_1=\infty$ and $g_1=\infty$.

\begin{proposition}
\label{propo_xiaorelcan}
With the notation as above, we have for any set of indexes $1\leq i_1\leq \ldots\leq i_{k-1}\leq i_k\leq n$
\begin{equation}
K_{X/C}^2\geq \sum_{j=1}^k (d_{i_j}+d_{i_{j+1}})(\mu_{i_j}-\mu_{i_{j+1}}),
\label{eq_xiaopos}
\end{equation}
where $i_{k+1}=n+1$,  $\mu_{n+1}$ is defined to be zero and $d_{n+1}=K_{X/C}.F=2g-2$.
\end{proposition}
\begin{proof}
This is an application of Xiao's Lemma (cf. \cite{xiao} Lemma 2).
\end{proof}

\begin{remark}
\label{rem_slopexi}
Notice that  Proposition \ref{propo_xiaorelcan} leads to a proof of slope inequality in our setting (recall that, throughout this Section, we are assuming that the relative canonical divisor is nef and all the subsequent quotients of the Harder-Narasimhan filtration are strongly semi-stable) for surface fibrations  in positive characteristic using the same argument exposed by Xiao in \cite{xiao} over the complex numbers.
A simple modification of the same Proposition (cf. \cite{gu2019slope}  Section 2.1 and \cite{contitesi} Section 4.1) extends the proof of the slope inequality for any relatively minimal surface fibration for which the subsequent quotients of the Harder-Narasimhan filtration are strongly semi-stable (cf. \cite{contitesi} Theorem 4.1.15). 
\end{remark}

Now consider the second multiplication map
\begin{equation}
\sigma\colon f_*{\omega_{X/C}}\otimes_{\O_C} f_*{\omega_{X/C}}\to f_*{\omega_{X/C}^2}
\label{eq_secondmult}
\end{equation}
and denote by $\F$ its image: by Max Noether's Theorem for integral projective Gorenstein non-hyperelliptic curves, cf. \cite{contiero} Theorem 1, it is generically surjective.
In particular we have $\deg(\F)\leq \deg(f_*\omega_{X/C}^2)$.

\begin{lemma}
\label{lemma_degs2f}
Let $f\colon X\to C$ be a non-hyperelliptic relatively minimal surface fibration.
We have 
\begin{equation}
K_{X/C}^2\geq \deg(\F)-\deg(f_*\omega_{X/C}).
\end{equation}
\end{lemma}

\begin{proof}
First of all we notice that $R^1f_*\omega_{X/C}^2=0$ where, for a coherent sheaf $E$, $R^1f_*E$ denotes the first higher pushforward of $E$: by the base change Theorem (cf. \cite{hart} Corollary III.12.9) we know that it is enough  to prove that $h^1(F,\restr{\omega_{X/C}^2}{F})=0$ for every fibre $F$.
By Riemann-Roch Theorem for (possibly singular) curves on a surface (cf. \cite{barth} Theorem III.3.1) and the adjunction Formula, this is equivalent to the vanishing of $h^0(F,-K_F)$.
Suppose that $h\in H^0(F,-K_F)$ is a section and let $A$ be the maximal closed subscheme of $F$ such that $\restr{h}{A}\equiv 0$, i.e. $h\in H^0(B,\O_B(-K_F-A))$ with $F=A+B$ (cf. Equation (4) of \cite{barth} Section II.1).
Because the fibration $f$ is relatively minimal, $B$ can contain only $(-2)$-curves: indeed on all the other components the degree of the restriction of $-K_F$ is strictly negative and the restriction of $h$ has to be identically zero.
Moreover $-K_F.A<0$ because otherwise $f$ would be a genus-one fibration, which is excluded by hypothesis, and $A$ and $B$ can not be linearly equivalent.

In particular $h$ defines the following exact sequence
\begin{equation}
0\to\O_B\to\O_B(-K_F-A)\to Q\to 0
\label{eq_exactrelativelyminimal}
\end{equation}
where $Q$ is a finitely supported sheaf, which means $\chi(Q)=h^0(B,Q)\geq0$.
By Riemann-Roch we have
\begin{equation}
-A.B=\deg(\O_B(-K_F-A))=\chi(\O_B(-K_F-A))-\chi(\O_B)=h^0(B,Q)\geq 0
\label{eq_rrabf}
\end{equation}
which, unless $B=0$, is in contradiction with Zariski's Lemma for surface fibrations (\cite{badescu} Corollary 2.6).
In particular $h$ has to be identically zero.

By the Leray spectral sequence (cf. \cite{badescu} page 103), relative duality (ibidem) and Riemann-Roch Theorem, we have that 
\begin{equation}
\begin{split}
\chi(\O_X)=\chi(f_*\O_X)&-\chi(R^1f_*\O_X)=\\
=-\deg(R^1f_*\O_X)-(\rk(R^1&f_*\O_X)-1)\chi(\O_C)=\\
=-\deg(R^1f_*\O_X/T(R^1f_*\O_X))&-l-(g-1)\chi(\O_C)=\\
=\deg(f_*\omega_{X/C})-(g-1)&\chi(\O_C)-l
\end{split}
\label{eq_pushforstructure}
\end{equation}
(where $l$ denotes the dimension of the torsion subsheaf $T(R^1f_*\O_X)$ of $R^1f_*\O_X$) and

\begin{equation}
\chi(\omega_{X/C}^2)=\chi(f_*\omega_{X/C}^2)=\deg(f_*\omega_{X/C}^2)+3(g-1)\chi(\O_C).
\label{eq_pushforsym}
\end{equation}

Using Riemann-Roch Theorem on $X$ we have
\begin{equation}
\begin{split}
\chi(\omega_{X/C}^2)=\chi(\O_X)&+K_{X/C}.(2K_{X/C}-K_X)=\\
=\chi(\O_X)+K_{X/C}^2-K_{X/C}.f^*K_C&=\chi(\O_X)+K_{X/C}^2+4(g-1)\chi(\O_C).\hspace{-10pt}
\end{split}
\label{eq_pushforrr}
\end{equation}

Combining Equations \ref{eq_pushforstructure}, \ref{eq_pushforsym} and \ref{eq_pushforrr} we get
\begin{equation}
\begin{split}
K_{X/C}^2=\chi(\omega_{X/C}^2)-\chi(\O_X)-4(g-1)\chi(\O_C)=\\
=\deg(f_*\omega_{X/C}^2)-\deg(f_*\omega_{X/C})+l\geq \deg(\F)-\deg(f_*\omega_{X/C})
\end{split}
\label{eq_pushforfin}
\end{equation}
and the Lemma is proven.
\end{proof}

Let's consider the following filtration for $\F$:
\begin{equation}
0=\F_0\subseteq \F_1\subseteq \F_2\subseteq\ldots\subseteq\F_{n-1}\subseteq\F_n=\F
\label{eq_filtrf}
\end{equation}
where $\F_i$ is the image of $\sigma$ restricted to $E_i\otimes E_i$.

\begin{lemma}
\label{lemma_rankfi}
In the above situation, if $r_i\geq 2$, we have
\begin{equation}
\rk(\F_i)\geq\min\{3r_i-3,2r_i+g_i-1\}\geq 2r_i-1
\label{eq_rkfi2}
\end{equation} 
and, if $\phi_i$ is birational, $\rk(\F_i)\geq 3r_i-3$. 
If $r_1=1$ we have $\rk(\F_1)\geq 2r_1-1=1$.
\end{lemma}

\begin{proof}
This follows easily from a similar bound for the dimension of the  image via the multiplication map of a (not necessarily complete) subcanonical linear system on an integral Gorenstein curve, a detailed proof of which can be found in \cite{contitesi} Corollary 1.6.9 (cf. also ibidem Remarks 1.6.10 and 1.6.11).
\end{proof}

Denote by $I_1=\{i\ |\ \delta_i=1\}$, $I_2=\{i\ |\ \delta_i=2,\ r_i\geq g_i+2\}$, $I_3=\{ i |\ \delta_i=2,\ r_i< g_i+2\}$  and $I_4=\{i\ |\ \delta_i\geq 3\}$. 
If $r_1=1$ we have defined $\delta_1=\infty$: in this case $1$ is an element of $I_4$.
Notice also that $n\in I_1$: indeed, by construction, we have that $\restr{\phi_n}{F}$ is the canonical morphism of $F$ which is an embedding (indeed for an integral Gorenstein non-hyperelliptic projective curve the canonical divisor defines an embedding cf. \cite{hartgen} Theorem 1.6), in particular $\phi_n$  is birational.

We rephrase Lemma \ref{lemma_rankfi} as
\begin{equation}
\rk(\F_i)\geq
\begin{cases}
2r_i-1 & i\in I_4;\\
3r_i-3 &  i\in I_3;\\
2r_i+g_i-1 &  i\in I_2;\\
3r_i-3 & i\in I_1.
\end{cases}
\label{eq_rkfi}
\end{equation}
where for $i\in I_4$ we are using the weaker inequality $\rk(\F_i)\geq 2r_i-1$, for $i \in I_3\cup I_2$ we are using the inequalities which come from the definition of $I_3$ and $I_2$ and for $i\in I_1$ we are using the stronger inequality $\rk(\F_i)\geq 3r_i-3$ which holds because for these indexes $\phi_i$ is birational.
Notice that we have $\rk(\F_n)=\rk(\F)=3r_n-3=3g-3$: this is clear because of the generic surjectivity of the morphism in \ref{eq_secondmult} and is needed for the application of Corollary \ref{coro_belloneq} in order to estimate the degree of $\F$.
Denote  by
\begin{equation}
\widetilde{r}_i=
\begin{cases}
2r_i-1 & i\in I_4\\
3r_i-3 & i\in I_3;\\
2r_i+g_i-1 & i\in I_2;\\
3r_i-3 & i\in I_1.
\end{cases}
\label{eq_ritilde}
\end{equation}

\begin{lemma}
\label{lemma_degfi}
In the same situation as in Lemma \ref{lemma_rankfi} we have
\begin{equation}
d_i\geq \delta_i\min\{2(r_i-1),r_i+g_i-1\}\geq\delta_i(r_i-1)
\label{eq_degfi2}
\end{equation}
and, if $\phi_i$ is birational, $d_i\geq 2(r_i-1)$.
\end{lemma}

\begin{proof}
Similarly as for Lemma \ref{lemma_rankfi}, the statement easily follows from a similar bound of the degree of a  (not necessarily complete) subcanonical linear system on an integral Gorenstein curve, a detailed proof of which can be found in \cite{contitesi} Corollary 1.6.7 (cf. also ibidem Remarks 1.6.8 and 1.6.11).
\end{proof}

We rephrase Lemma \ref{lemma_degfi} as
\begin{equation}
\label{eq_casesdegrk}
d_i\geq 
\begin{cases}
\delta_i(r_i-1)\geq3(r_i-1) & i\in I_4;\\
2(2r_i-2) & i\in I_3;\\
2(r_i+g_i-1) & i\in I_2;\\
2(r_i-1) & i\in I_1.
\end{cases}
\end{equation}
where for $i\in I_4$ we are using the weaker inequality $d_i\geq \delta_i(r_i-1)$, for $i\in I_3\cup I_2$ we are using the inequalities which come from the definition of $I_3$ and $I_2$ and for $i\in I_1$ we are using the stronger inequality $d_i\geq 2r_i-2$ which holds because for these indexes $\phi_i$ is birational.
Observe that this inequality is still valid when $r_1=1$ and hence $d_1=0$.

Thanks to Equation \ref{eq_tensorstrong} and the fact that $\mu_{min}(E)$ is characterized as the minial slope of a subbundle of $E$, we know that 
\begin{equation}
\label{eq_degf2}
\mu_{min}({\F_i})\geq\mu_{min}(E_i\otimes E_i)= 2\mu_i.
\end{equation}

Now, thanks to Corollary \ref{coro_belloneq}, where, in the notation of that Corollary, $\widetilde{\mu}_i=2\mu_i$ and $\widetilde{r}_i$ as defined in Equation \ref{eq_ritilde},  we get
\begin{equation}
\begin{split}
\deg(\F)\geq \sum_{i\in I_4}(4r_i-2)(\mu_i-\mu_{i+1})+\sum_{i\in I_3}(6r_i-6)(\mu_i-\mu_{i+1})+\\
+\sum_{i\in I_2}(4r_i+2g_i-2)(\mu_i-\mu_{i+1})+\sum_{i\in I_1}(6r_i-6)(\mu_i-\mu_{i+1}).
\end{split}
\label{eq_nuovanew10}
\end{equation}

\begin{theorem}[cf. Theorem 1.2 \cite{zuo} for 1. over the complex numbers and Theorem 3.1 \cite{gusunzhou} for 2. with $0\leq c\leq\frac{1}{3}$]
\label{teo_slopegusunzhou}
Let $f\colon X\to C$ be a non-hyperelliptic  surface fibration such that $K_{X/C}$ is nef and all the quotients appearing in the Harder-Narasimhan filtration of $f_*\omega_{X/C}$ are strongly semi-stable.
\begin{enumerate}
	\item Assume that there exists an integer $\delta>1$ such that $\delta_i=1$ or $\delta_i>\delta$ for all $i$: then
	\begin{equation}
	K_{X/C}^2\geq \Bigl(5-\frac{1}{\delta}\Bigr)\frac{g-1}{g+2}\deg(f_*\omega_{X/C}).
	\label{eq_gusunzhou1}
	\end{equation}
	\item Assume that there exists a constant $0\leq c\leq\frac{1}{2}$ such that $g_i\geq cg$ whenever $\delta_i=2$: then
	\begin{equation}
	K_{X/C}^2\geq (4+c)\frac{g-1}{g+2}\deg(f_*\omega_{X/C}).
	\label{eq_gusunzhou2}
	\end{equation}
\end{enumerate}
\end{theorem}

\begin{proof}
First of all we notice that if $\deg(f_*\omega_{X/C})$ is non-positive, the thesis is trivial in both cases.
Hence we may assume that $\deg(f_*\omega_{X/C})>0$ from which, essentially by definition, we derive that $\mu_1>0$.

By Lemma \ref{lemma_degs2f} and Equation \ref{lemma_deghn} we have
\begin{equation}
K_{X/C}^2\geq \deg(\F)-\deg(f_*\omega_{X/C})=\deg(\F)-\sum_{i=1}^nr_i(\mu_i-\mu_{i+1}).
\label{eq_nuovanew1}
\end{equation}

Combining Equations \ref{eq_nuovanew1} and \ref{eq_nuovanew10}, we get
\begin{equation}
\begin{split}
K_{X/C}^2\geq \sum_{i\in I_4}(3r_i-2)(\mu_i-\mu_{i+1})+\sum_{i\in I_3}(5r_i-6)(\mu_i-\mu_{i+1})+\\
+\sum_{i\in I_2}(3r_i+2g_i-2)(\mu_i-\mu_{i+1})+\sum_{i\in I_1}(5r_i-6)(\mu_i-\mu_{i+1}).
\end{split}
\label{eq_slopezhou1}
\end{equation}

Case 1 can be proved as over the complex numbers (cf. Theorem 1.2 \cite{zuo}) or with a similar argument as for case 2 (cf. Theorem 4.2.2(2) \cite{contitesi}, notice that in this case $I_2$ and $I_3$ are empty).

Now consider the second part of the Theorem.
By  Equation \ref{eq_casesdegrk} we get
\begin{equation}
\label{eq_boh1}
d_i+d_{i+1}\geq 2d_i\geq 6(r_i-1)
\end{equation}
for $i\in I_4$ (this also works for $r_1=1$).
Similarly we obtain 
\begin{equation}
d_i+d_{i+1}\geq 2(2r_i-2)+2(r_{i+1}-1)\geq6r_i-4\geq 4r_i-4
\end{equation}
for $i\in I_3$,
\begin{equation}
d_i+d_{i+1}\geq 2(r_i+g_i-1)+2(r_{i+1}-1)\geq 4r_i+2g_i-2
\end{equation}
for $i\in I_2$ and
\begin{equation}
d_i+d_{i+1}\geq 2(r_i-1)+2(r_{i+1}-1)\geq 4r_i-2
\end{equation}
for $i\in I_1\setminus\{n\}$.
Moreover $d_n+d_{n+1}=4g-4=4r_n-4$.
Notice that for $d_{i+1}$ we are always using the weaker inequality $d_{i+1}\geq 2r_{i+1}-2\geq 2r_i$ of \ref{eq_casesdegrk} which holds for every index $i$ except for $n$ which we are excluding from $I_1$ because it is the unique index $i$ for which $r_{i+1}>r_i$ does not hold.

Combining these equations with Proposition \ref{propo_xiaorelcan} with set of indexes $1,\ldots,n$, we get
\begin{equation}
\begin{split}
K_{X/C}^2\geq \sum_{i\in I_4}(6r_i-6)(\mu_i-\mu_{i+1})+\sum_{i\in I_3\cup\{n\}}(4r_i-4)(\mu_i-\mu_{i+1})\\
+\sum_{i\in I_2}(4r_i+2g_i-2)(\mu_i-\mu_{i+1})+\sum_{i\in I_1\setminus\{n\}}(4r_i-2)(\mu_i-\mu_{i+1}).
\end{split}
\label{eq_slopezhou2}
\end{equation} 

Calculating $c(\ref{eq_slopezhou1})+(1-c)(\ref{eq_slopezhou2})$ we have (recall that $g=r_n\geq r_i$ and $g_i\geq cg$)
\begin{equation}
\label{eq_slopezhou4}
\begin{split}
K_{X/C}^2\geq (4+c)\sum_{i=1}^{n}r_i(\mu_i-\mu_{i+1})+(2-4c)\sum_{i\in I_4}r_i(\mu_i-\mu_{i+1})\\
-(6-4c)\sum_{i\in I_4}(\mu_i-\mu_{i+1})-(4+2c)\sum_{i\in I_3}(\mu_i-\mu_{i+1})\\
-2\sum_{i\in I_2}(\mu_i-\mu_{i+1})-(2+4c)\sum_{i\in I_1\setminus\{n\}}(\mu_i-\mu_{i+1})-(4+2c)\mu_n.
\end{split}
\end{equation}

We notice that, for $1/3\leq c\leq 1/2$, we have $4+2c\geq 6-4c$, $4+2c\geq 2$, $4+2c\geq 2+4c$ and $2-4c\geq 0$: in particular
\begin{equation}
\begin{split}
K_{X/C}^2\geq (4+c)\deg(f_*\omega_{X/C})&-(4+2c)(\mu_1-\mu_n)-(4+2c)\mu_n\geq\\
\geq(4+c)\deg(f_*&\omega_{X/C})-(4+2c)\mu_1
\end{split}
\end{equation}
(we are using that $\mu_i-\mu_{i+1}\geq 0$ for $i\neq n$).

Combining this equation with $K_{X/C}^2\geq(2g-2)\mu_1$ (here we are using Proposition \ref{propo_xiaorelcan} with set of indexes $\{1\}$) we obtain
\begin{equation}
K_{X/C}^2\geq(4+c)\frac{g-1}{g+1+c}\deg(f_*\omega_{X/C})\geq(4+c)\frac{g-1}{g+2}\deg(f_*\omega_{X/C})
\end{equation}
where we are using the assumption $\deg(f_*\omega_{X/C})>0$ made at the beginning of this proof.
Hence we have proven 2 for a constant $1/3\leq c\leq 1/2$.

One can substitute Equation \ref{eq_boh1} by
\begin{equation}
\label{eq__boh1}
d_i+d_{i+1}\geq 5r_i-4
\end{equation}
for $i\in I_4$: indeed, for $i\in I_4$, if $r_i>1$ we clearly have $6r_i-6\geq 5r_i-4$, while if $r_1=1$ we have that $d_1+d_2=d_2\geq 1=5r_1-4$.

Arguing as above just substituting inequality \ref{eq_boh1} with inequality \ref{eq__boh1}, we obtain the thesis of the Theorem for a constant $0\leq c\leq\frac{1}{3}$ (cf. also Theorem 3.1 \cite{gusunzhou}).
%In this case Proposition   \ref{propo_xiaorelcan}, with set of indexes $1,\ldots,n$, gives
%\begin{equation}
%\begin{split}
%K_{X/C}^2\geq \sum_{i\in I_4}(5r_i-4)(\mu_i-\mu_{i+1})+\sum_{i\in I_3\cup\{n\}}(4r_i-4)(\mu_i-\mu_{i+1})\\
%+\sum_{i\in I_2}(4r_i+2g_i-2)(\mu_i-\mu_{i+1})+\sum_{i\in I_1\setminus\{n\}}(4r_i-2)(\mu_i-\mu_{i+1})
%\end{split}
%\label{eq_slopezhou5}
%\end{equation} 
%and  $c(\ref{eq_slopezhou1})+(1-c)(\ref{eq_slopezhou5})$ becomes
%\begin{equation}
%\label{eq_slopezhou6}
%\begin{split}
%K_{X/C}^2\geq (4+c)\sum_{i=1}^{n}r_i(\mu_i-\mu_{i+1})+(1-3c)\sum_{i\in I_4}r_i(\mu_i-\mu_{i+1})\\
%-(4-2c)\sum_{i\in I_4}(\mu_i-\mu_{i+1})-(4+2c)\sum_{i\in I_3}(\mu_i-\mu_{i+1})\\
%-2\sum_{i\in I_2}(\mu_i-\mu_{i+1})-(2+4c)\sum_{i\in I_1\setminus\{n\}}(\mu_i-\mu_{i+1})-(4+2c)\mu_n.
%\end{split}
%\end{equation}
%
%We notice that, for $0\leq c\leq 1/3$, we have $4+2c\geq 4-2c$, $4+2c\geq 2$, $4+2c\geq 2+4c$ and $1-3c\geq 0$: in particular
%\begin{equation}
%\begin{split}
%K_{X/C}^2\geq (4+c)\deg(f_*\omega_{X/C})&-(4+2c)(\mu_1-\mu_n)-(4+2c)\mu_n\geq\\
%\geq(4+c)\deg(f_*&\omega_{X/C})-(4+2c)\mu_1.
%\end{split}
%\end{equation}
%Combining this equation with $K_{X/C}^2\geq(2g-2)\mu_1$ (here we are using Proposition \ref{propo_xiaorelcan} with set of indexes $\{1\}$) we obtain
%\begin{equation}
%K_{X/C}^2\geq(4+c)\frac{g-1}{g+1+c}\deg(f_*\omega_{X/C})\geq(4+c)\frac{g-1}{g+2}\deg(f_*\omega_{X/C})
%\end{equation}
%where we are using the assumption $\deg(f_*\omega_{X/C})>0$ made at the beginning of this proof.
%This concludes the proof.
\end{proof}

\begin{corollary}[cf. \cite{lu} Theorem 2.1 for the same result over the complex numbers]
Let $f\colon X\to C$ a surface fibration  such that $K_{X/C}$ is nef and all the quotients appearing in the Harder-Narasimhan filtration of $f_*\omega_{X/C}$ are strongly semi-stable and suppose that $K_{X/C}^2<\frac{9(g-1)}{2(g+2)}\deg(f_*\omega_{X/C})$. 
Then $f$ factors through a morphism $\pi\colon X\to Y$ of degree $2$, i.e. the following diagram commutes
\begin{equation}
\begin{tikzcd}
X\arrow{rr}{\pi}\arrow{dr}{f} & & Y\arrow{ld}{}\\
& C. & 
\end{tikzcd}
\end{equation}
\label{coro_slopegusunzhou2}
\end{corollary}

\begin{proof}
Suppose by contradiction that there is not any morphism $f\colon X\to Y$ of this type: then we can apply part 1 of Theorem \ref{teo_slopegusunzhou} with $\delta=2$ and we get a contradiction.\qedhere
\end{proof}

\section{Factors of the Albanese morphism and refined Severi inequalities}
\label{sec_albanese}

In this Section we are going to prove Theorems \ref{teo_mio1} and \ref{teo_mio2}.
In particular the latter is the starting point (as done by Lu and Zuo in \cite{zuo} for the same result over the complex numbers) for the characterization of surfaces close to the Severi lines we will give in the next Sections.
We recall the results obtained in \cite{gusunzhou} following their proofs step by step and the main results of these section are refinement of results there obtained.
More in detail, Theorem \ref{teo_mio1} is a generalization of Theorem 5.5 loc. cit. (there, the presence of an involution which is a factor of the Albanese morphism is proved, without emphatizing it, only under the condition that certain finite morphisms of degree two are separable which is clearly always satisfied if the characteristic of the ground field is different from $2$) and \ref{teo_mio2} is a generalization of Theorem 5.14 loc. cit. (where the generalization is given by the fact that the constant  $c$ is contained in $[0,\frac{1}{2}]$ instead of $[0,\frac{1}{3}]$).

First we recall the construction of the surface fibrations over the projective line associated with $X$ presented by Pardini in \cite{par_sev} where they are used to prove Severi inequality.

Let $X$ be a minimal surface of general type with maximal Albanese dimension, let $\alb_X\colon X\to\Alb(X)$ be its Albanese morphism, denote by $q=\dim(\Alb(X))$ and let $L$ be a very ample line bundle on $\Alb(X)$ with $h^0(\Alb_X,L)\gg 0$.
For any integer $d$ coprime with the characteristic, consider the multiplication by $d$ morphism $\mu_d\colon \Alb(X)\to\Alb(X)$: it is known that (\cite{mumford} Proposition at the end of Section 6 and Corollary 1 Section 7)  it is an \'etale morphism of degree $d^{2q}$ and that (ibidem Section 8.(iv))
\begin{equation}
\mu_d^*(L)\sim_{num}d^2L.
\label{eq_'etalb}
\end{equation}

Now consider the following Cartesian diagram:
\begin{equation}
\label{eq_cartalb}
\begin{tikzcd}
X_d\arrow{r}{\nu_d} \arrow[dr, phantom, "\square"] \arrow{d}{a_d} & X\arrow{d}{\alb_X}\\
\Alb(X)\arrow{r}{\mu_d} & \Alb(X).
\end{tikzcd}
\end{equation}
It is clear that $\nu_d$ is \'etale and that $X_d$ is a connected minimal surface of general type of maximal Albanese dimension with invariants:
\begin{equation}
K_{X_d}^2=d^{2q}K_X^2,\qquad \chi(\O_{X_d})=d^{2q}\chi(\O_X).
\label{eq_invcartalb}
\end{equation}

Denote by $L_X=\alb_X^*L$ and $L_{X_d}=a_d^*L$: we have $L_{X_d}\sim_{num}d^{-2}\nu_d^*L_X$, in particular
\begin{equation}
L_{X_d}^2=d^{2q-4}L_X^2, \qquad K_{X_d}.L_{X_d}=d^{2q-2}K_X.L_X.
\label{eq_multcartalb}
\end{equation}

It is not difficult to prove that for a general element $(D_1,D_2)\in a_d^*|L|\times a_d^*|L|$ we have that $D_1$ and $D_2$ are integral, non-hyperelliptic and the intersection points of $D_1$ with $D_2$ lie in their smooth locus (cf. \cite{gusunzhou} Lemmas 4.1 and 5.3 and Corollary 4.2).
Consider the rational map $\phi'_d\colon X_d\dashrightarrow \P^1$ induced by the linear system generated by $D_1$ and $D_2$ and let $\phi_d\colon \widetilde{X}_d\to \P^1$ be its resolution of indeterminacy  obtained by successive blow-ups of the intersection points of the strict transforms of $D_1$ and $D_2$ and denote by $\pi_d\colon \widetilde{X}_d\to X_d$ the natural morphism.

\begin{lemma}[\cite{gusunzhou} Proposition 4.4]
\label{lem_surfib}
In the situation above, denote by $p_1,\ldots,p_r$ the intersection points of $D_1$ with $D_2$. 
Then we have
\begin{enumerate}
	\item $K_{\widetilde{X}_d}=\pi_d^*K_{X_d}+\sum_{i=1}^{r}(E_{i_1}+2E_{i_2}+\ldots+m_iE_{i_{m_i}})$ where $E_{i_j}$ denotes the strict transform of the exceptional divisor of the $j$-th blow-up over $p_i$;
	\item $E_{i_j}$ are $(-2)$-curves for $j<m_i$ and $E_{i_{m_i}}$ is a $(-1)$-curve;
	\item the strict transforms $\widetilde{D}_1$ and $\widetilde{D}_2$ of $D_1$ and $D_2$ are fibres of $\phi_d$, in particular the fibration has connected fibres;
	\item $E_{i_{m_i}}$ is a section of $\phi_d$ and $E_{i_j}$ is vertical for every $j<m_i$, in particular the fibration has no multiple fibres;
	\item the relative canonical bundle $K_{\widetilde{X}_d}-\phi_d^*K_{\P^1}$ is nef.
\end{enumerate}
\end{lemma}

\begin{remark}
\label{rem_notorsionhigher}
Let $F$ be  a fibre of $\phi_d$.
For every decomposition $F=A+B$ the intersection $A.B$ is greater than $0$: indeed $0=F^2=A^2+2A.B+B^2<2A.B$ where the last inequality follows by the Zariski's Lemma for fibrations (\cite{badescu} Corollary 2.6).
We have thus shown that $F$ is $1$-connected (cf. \cite{barth} Section II.12), in particular $h^0(F,\O_F)=1$ for every fibre $F$ of $\phi_d$ (cf. ibidem Lemma II.12.2).
Combining this with the base change Theorem (\cite{hart} Corollary III.12.9) we can derive that $R^1f_{d*}\O_{\widetilde{X}_d}$ is locally free: This is always true for a surface fibration over the complex numbers (cf. \cite{barth} Corollary III.11.2), but it is not always the case for a surface fibration in positive characteristic (cf. \cite{badescu} Chapter 7). 
\end{remark}

Because the blown-up points on $X_d$ lie above the smooth locus of $D_1$, its arithmetic genus does not vary: in particular the arithmetic genus of a fibre of $\phi_d$ is
\begin{equation}
g_d=1+\frac{K_{X_d}.L_{X_d}+L_{X_d}^2}{2}=1+\frac{d^{2q-2}K_X.L_X+d^{2q-4}L_X^2}{2}.
\label{eq_genfibr}
\end{equation}
Notice that, clearly $K_X.L_X>0$ because $X$ is minimal and $L_X$ is the pull-back of an ample divisor via the Albanese morphism, hence it can not be a union of $-2$-curves: for $d\to\infty$ we have $g_d\to\infty$.

Moreover
\begin{equation}
\begin{split}
K_{\widetilde{X}_d/\P^1}^2=K_{\widetilde{X}_d}^2+4F.K_{\widetilde{X}_d}&=K_{X_d}^2+E^2+8(g_d-1)=\\
=K_{X_d}^2-L_{X_d}^2+4K&_{X_d}.L_{X_d}+4L_{X_d}^2=\\
=d^{2q}K_X^2+4d^{2q-2}K_X.&L_X+3d^{2q-4}L_X^2
\end{split}
\label{eq_cansqfib}
\end{equation}
where $E=\sum_{i=1}^r(E_{i_1}+2E_{i_2}+\ldots+E_{i_{m_i}})$.

Combining Equation \ref{eq_pushforstructure} with Remark \ref{rem_notorsionhigher} we get
\begin{equation}
\begin{split}
\deg((\phi_d)_*\omega_{\widetilde{X}_d/\P^1})=\chi(\O_{\widetilde{X}_d})+g_d-1=\\
=d^{2q}\chi(\O_X)+\frac{d^{2q-2}K_X.L_X+d^{2q-4}L_X^2}{2}.
\end{split}
\label{eq_deghigherpushf}
\end{equation}

Notice that, thanks to \cite{gu2019slope} Corollary 3.4, we have that $\chi(\O_X)>0$ which implies that $\deg((\phi_d)_*\omega_{\widetilde{X}_d/\P^1})>0$ for $d\gg 0$.
In particular we can define the slope of $\phi_d$ and we have that
\begin{equation}
\lambda(\phi_d)=\frac{2d^{2q}K_X^2+8d^{2q-2}K_X.L_X+6d^{2q-4}L_X^2}{2d^{2q}\chi(\O_X)+d^{2q-2}K_X.L_X+d^{2q-4}L_X^2}\xrightarrow{d\to\infty}\frac{K_X^2}{\chi(\O_X)}.
\label{eq_limitslope}
\end{equation}

In particular we have constructed a sequence of surface fibrations $\phi_d\colon\widetilde{X}_d\to\P^1$ which satisfies:
\begin{equation}
\label{eq_maverso}
\begin{split}
	\lim_{d\to\infty}g_d&=\infty;\\
	\lim_{d\to\infty}\lambda(\phi_d)&=\frac{K_X^2}{\chi(\O_X)}.
	\end{split}
\end{equation}

\begin{remark}
It is not difficult to use these surface fibrations and the Slope inequality (cf. \ref{rem_slopexi} and recall that by Lemma \ref{lem_surfib}(5) $\phi_d\colon\widetilde{X}_d\to\P^1$ have nef relative canonical bundle) to prove Severi inequality for a minimal surface of general type (i.e. $K_X^2\geq 4\chi(\O_X)$) using the argument of Pardini in \cite{par_sev}.
\end{remark}

\begin{theorem}
Let $X$ be a surface of general type and let $\phi'_d\colon X_d\dashrightarrow\P^1$ be the surface fibration constructed above for $d\gg 0$ and coprime with the characteristic of the ground field.
If $K_X^2<\frac{9}{2}\chi(\O_X)$ then there exists a rational map $f_d\colon X_d\dashrightarrow Y_d$ to a surface $Y_d$ of degree $2$ such that
\begin{equation}
\begin{tikzcd}
X_d\arrow[dashed]{dr}{\phi_d}\arrow[dashed]{rr}{f_d} & & Y_d\arrow{dl}{\psi_d}\\
& \P^1 &
\end{tikzcd}
\end{equation}
commutes.
Moreover, if $f_d$ is separable, it is then induces by an involution $\sigma_d$ of $X_d$.
\label{teo_lualb2}
\end{theorem}

\begin{proof}
This follows from Corollary \ref{coro_slopegusunzhou2} using the same argument as in \cite{zuo} Theorem 2.2 (there the same result is proved over the complex numbers).
\end{proof}

\begin{lemma}[cf. \cite{lu} Lemma 3.2]
Suppose that $f_d$ is separable for a general pencil in $a_d^*|L|$ and assume, as usual, that $d\gg 0$ and coprime with the characteristic of the ground field.
Then, for a general pencil in $a_d^*|L|$, the involution $\sigma_d$ of $X_d$ corresponding to $f_d$ satisfies $a_d\circ\sigma_d=a_d$.
\end{lemma}

\begin{proof}
The proof given for Lemma 3.2 \cite{lu} is characteristic-free if we assume that $f_d$ is separable.
\end{proof}

\begin{proposition}[\cite{gusunzhou} Proposition 5.10]
\label{propo_gusunzhoufactalb}
Suppose that for infinitely many prime numbers $l$ we have a suitable subpencil of $a_l^*|L|$ such that the corresponding degree-two morphism $f_l$ is separable and let $\sigma_l$ be its corresponding involution.
Then $f_l$ descends to a separable degree-two morphism $f\colon X\to Y$ induced by an involution $\sigma$ of $X$, such that $\alb_X\circ\sigma=\alb_X$.
\label{propo_infiniteseparable2}
\end{proposition}

This last Proposition is the main ingredient in order to prove Theorem \ref{teo_mio1}: what remains to prove is that, if all but finitely many $f_d$ are inseparable (in particular the characteristic of the ground field is two) and the Albanese morphism of $X$ is separable (if $\alb_X$ is inseparable the proof is almost immediate, as we will see), then $K_X^2\geq\frac{9}{2}\chi(\O_X)$ which will be proved using Theorem \ref{teo_slopegusunzhou}(2).

With the notation as in Theorem \ref{teo_lualb2}, let $\widetilde{Y}_d$ be the minimal model of the resolution of singularities of $Y_d$ and let $g_d'$ be the arithmetic genus of a general fibre of $\psi_d\colon \widetilde{Y}_d\to\P^1$.
We are going to give an estimate of $g_d'/g_d$ in order to apply  Theorem \ref{teo_slopegusunzhou}.
%Notice that we can use $g_d'$ instead of the geometric genus of the fibre thanks to Remark \ref{rem_slopezhou}.

\begin{proposition}[cf. \cite{gusunzhou} Propositions 5.2 and 5.13] 
Suppose that $\alb_X\colon X\to\Alb(X)$ is a separable morphism, that the characteristic of the ground field is $2$ and that for a general choice of a pencil in $a_d^*|L|$ for all but finitely many $d$ coprime with $2$, we have that the corresponding rational map $f_d\colon X_d\dashrightarrow Y_d$ is inseparable.
Then for a suitable pencil we have
\begin{equation}
\lim_{d\to\infty}\frac{g_d'}{g_d}\geq\frac{1}{2}.
\label{nonmiserve}
\end{equation}
\label{propo_nonmiserve}
\end{proposition}

\begin{proof}
By \cite{gu_2016} Proposition 2.2 (cf. also \cite{gusunzhou} proof of Proposition 5.13) we have that $g_d'=g_d-\frac{1}{4}\rk\Bigl((\phi_d)_*\bigl(T(\Omega_{\widetilde{X}_d/\P^1}^1)\bigr)\Bigr)$.
So, using Lemma 5.8 of \cite{gusunzhou} and Equation \ref{eq_genfibr} we get
\begin{equation}
\label{eq_naltra}
\begin{split}
g_d'\geq g_d-\frac{L_{X_d}.c_1\bigl(\Omega_{X_d/\Alb(X)}^1\bigr)}{4}=&\\
=\frac{d^{2q-2}K_X.L_X+d^{2q-4}L_X^2}{2}-\frac{d^{2q-2}c_1(\Omega_{X/\Alb(X)}^1).L_X}{4}&+1=\\
=\frac{d^{2q-2}(2K_X-c_1(\Omega_{X/\Alb(X)}^1)).L_X+2d^{2q-4}L_X^2}{4}&+1.
\end{split}
\end{equation}
By \cite{gusunzhou} Proposition 5.2, we have  
\begin{equation}
\label{eq_gusunzhouc1}
(2K_X-c_1(\Omega_{X/\Alb(X)}^1)).L_X\geq K_X.L_X.
\end{equation}
Hence 
\begin{equation}
\lim_{d\to\infty}\frac{g_d'}{g_d}\geq\lim_{d\to\infty}\frac{2(d^{2q-2}K_X.L_X+2d^{2q-4}L_X^2)+8}{4(d^{2q-2}K_X.L_X+2d^{2q-4}L_X^2)+8}=\frac{1}{2}
\label{eq}
\end{equation}
and the Proposition is proven.
\end{proof}

\begin{proof}[Proof of Theorem \ref{teo_mio1}]
Suppose that the characteristic of the ground field is different from $2$.
Then the Theorem directly follows from Proposition \ref{propo_gusunzhoufactalb}.

Suppose that the characteristic of the ground field is two.
If the Albanese morphism is inseparable, this is exactly Lemma \ref{lemma_albinsfact}.

Now suppose that the Albanese morphism is separable.
Thanks to Proposition \ref{propo_nonmiserve} and part two of Theorem \ref{teo_slopegusunzhou}, we can derive that if all but finitely many $f_d$ are inseparable, then $K_X^2\geq\frac{9}{2}\chi(\O_X)$ which is excluded by the hypothesis.
Hence, also in this case the thesis is a consequence of Proposition \ref{propo_gusunzhoufactalb}.
\end{proof}

Concerning Theorem \ref{teo_mio2}, observe that this is the same as Theorem 5.14 of \cite{gusunzhou} except that one take $\frac{1}{2}$ instead of $\frac{1}{3}$. This improvement can be done because we have extended Theorem 3.1(2) loc. cit. in Theorem \ref{teo_slopegusunzhou}(2), where the generalization is given by the fact that the constant $c$ can take value up to $\frac{1}{2}$ instead of $\frac{1}{3}$.
Here we briefly recall the argument exposed in \cite{gusunzhou} in order to derive Theorem \ref{teo_mio2} from Theorem \ref{teo_slopegusunzhou}(2).

Let $X$ be a minimal surface of general type with maximal Albanese dimension and denote by $\alb_X\colon X\to \Alb(X)$ its Albanese morphism.
Notice that there are finitely many rational maps of degree two $\pi_i\colon X\dashrightarrow Y_i$ with $Y_i$ smooth and minimal such that there exists a morphism $b_i\colon Y_i\to \Alb(X)$ for which $b_i\circ\pi_i=\alb_X$.
Indeed the separable morphisms $\pi_i$ are in natural correspondence with a subset of the group of automorphism of $X$ which is finite, while there exists a unique inseparable morphism $\pi_i$ of degree two when $\alb_X$ is inseparable and the characteristic of the ground field is two and it is uniquely determined by the 1-foliation $\T_{X/\Alb(X)}$ (cf. Lemma \ref{lemma_albinsfact}).

\begin{definition}
Let $\pi_i\colon X\dashrightarrow Y_i$ as above and let $L$ be a very ample line bundle on $\Alb(X)$.
Denote by $L_X=\alb_X^*L$ and $L_{Y_i}=b_i^*L$.
We define 
\begin{equation}
c_i(X,L)=\frac{K_{Y_i}.L_{Y_i}}{K_X.L_X}.
\label{eq_cixl}
\end{equation}
If the characteristic of the ground field is two and $\alb_X$ is separable we define
\begin{equation}
c_0(X,L)=\frac{(2K_X-c_1(\Omega_{X/\Alb(X)}^1)).L_X}{2K_X.L_X}.
\label{eq_c0xl}
\end{equation}

Now take 
\begin{equation}
c(X,L)=\min_{i>0}\{c_i(X,L)\}
\label{eq_cxl}
\end{equation}
 and
\begin{equation}
c(X)=\sup\{c(X,L)\}
\label{eq_cx}
\end{equation}
for all possible very ample line bundles $L$ on $\Alb(X)$.
\label{def_cxl}
\end{definition}

It is clear by definition that $c_i(X,L)=0$ if and only if $Y_i$ is an Abelian surface.

\begin{remark}
Observe that the definition of all $c_i(X,L)$ and, consequently the definition of $c(X,L)$, does not change if we substitute $L$ with a multiple of $L$.
\label{rem_cxl}
\end{remark}

\begin{proof}[Proof of Theorem \ref{teo_mio2}]
So far we have constructed the following commutative diagram
\begin{equation}
\begin{tikzcd}
& X_d\arrow[dashed]{dd}{f_d}\arrow[dashed]{dl}{\phi_d}\arrow{dr}{a_d}& \\
\P^1 & & \Alb(X)\\
& \widetilde{Y}_d \arrow{ul}{\psi_d}\arrow{ur}{b_d}& 
\end{tikzcd}
\label{eq_nonsocomexcl}
\end{equation}
where $a_d\colon X_d\to \Alb(X)$ is the pull-back via the multiplication by an integer $d$ coprime with the characteristic of the ground field of the Albanese morphism $\alb_X\colon X\to \Alb(X)$ (cf. Equation \ref{eq_invcartalb}), $f_d$ is a degree-two rational map and $\widetilde{Y}_d$ can be assumed to be smooth and minimal.
We have defined $g_d$ to be the arithmetic genus of a general fibre of $\phi_d$ and $g'_d$ to be the arithemtic genus of a general fibre of $\psi_d$.
Recall also that this construction depends on the choice of  a pencil in $a_d^*|L|$ where $L$ is a very ample line bundle on $\Alb(X)$.
Arguing in a similar way as for the fibres of $X_d\dashrightarrow\P^1$ (cf. \cite{gusunzhou}  Lemma 4.1 and Corollary 4.2), we can restrict the choice of the pencil so that the general fibre of $\psi_d$ is integral too.

If the characteristic of the ground field is $2$ and $\Alb(X)$ is inseparable or if the the hypotheses of Proposition \ref{propo_gusunzhoufactalb} are satisfied, then the right hand side of diagram \ref{eq_nonsocomexcl} descends to
\begin{equation}
\begin{tikzcd}
X\arrow[dashed]{rr}{\pi_i}\arrow{dr}{\alb_X}& & Y_i\arrow{dl}{b_i}\\
& \Alb(X) &
\label{eq_nonsoxcl}
\end{tikzcd}
\end{equation}
where the horizontal arrow has to be one of the $\pi_i$ defined above.
In this case, a formula similar to the one for $g_d$ (Equation \ref{eq_genfibr}) holds also for $g_d'$, that is
\begin{equation}
g_d'=1+\frac{d^{2q-2}K_{Y_i}.L_{Y_i}+d^{2q-4}L_{Y_i}^2}{2}
\label{eq_genfibr'}
\end{equation}
Combining Equations \ref{eq_genfibr} and \ref{eq_genfibr'}, we easily derive that 
\begin{equation}
\lim_{d\to\infty}\frac{g_d'}{g_d}=c_i(X,L);
\label{eq_ultimaboh}
\end{equation}
in particular, for $d$ sufficiently large, Theorem \ref{teo_slopegusunzhou}(2) applies and, for $d\to\infty$, we get (cf. Equation \ref{eq_maverso})
\begin{equation}
K_X^2\geq \Bigl(4+\min \{c_i(X,L),\frac{1}{2}\}\Bigr)\chi(\O_X).
\label{eq_mahult}
\end{equation}

If the characteristic of the ground field is two and $\alb_X$ is a separable morphism and $f_d$ does not descend to a degree two morphism $X\to Y_i$, combining Equations \ref{eq_genfibr} and \ref{eq_naltra} we get
\begin{equation}
\lim_{d\to\infty}\frac{g_d'}{g_d}\geq c_0(X,L).
\label{eq_anctra}
\end{equation}
Notice that Equation \ref{eq_gusunzhouc1} states that $c_0(X,L)\geq\frac{1}{2}$.
Again, applying Theorem \ref{teo_slopegusunzhou}(2), we get
\begin{equation}
K_X^2\geq \Bigl(4+\min \{c_0(X,L),\frac{1}{2}\}\Bigr)\chi(\O_X)=\frac{9}{2}\chi(\O_X)
\label{eq_mahult'}
\end{equation}
which concludes the proof.
\end{proof}

\section{Severi type inequalities in characteristic different from $2$}
\label{sec_severi}

In this Section we are going to prove Theorem \ref{teo_mio3} so that throughout this section all varieties are considered over an algebraically closed field of characteristic different from $2$.
The proof follows the outline of the work of Lu and Zuo in \cite{lu} and the author in \cite{conti} over the complex numbers.
Because most of the arguments shown there rely on the theory of double covers, everything proceeds smoothly, up to some minor changes, when the characteristic of the ground field is different from two.
In particular in this case we have to pay a particular attention to the fact that in positive characteristic it may happen that for a surface $X$ we have $q(X)\neq q'(X)$.

The starting point of the proof is Theorem \ref{teo_mio1}: as stated there, the condition $K_X^2<\frac{9}{2}\chi(\O_X)$ assures that there exists  an involution $i\colon X\to X$ with respect to which the Albanese morphism $\alb_X$ is stable (or, equivalently, $\alb_X$ is composed with $i$), i.e.
\[
\begin{tikzcd}
X \arrow{rr}{i}\arrow[swap]{dr}{\alb_X}& & X\arrow{dl}{\alb_X} \\
 & \Alb(X) &
\end{tikzcd}
\]
is commutative.

\begin{remark}[cf. \cite{conti} Remark 4.1 for more details]
Notice that the condition "$\alb_X$ is composed with an involution" is necessary. Otherwise the surface $X=A\times B$ with $g(A)=2\leq g(B)$ gives a counter example.

Notice also that, at least for Theorem \ref{teo_mio2},  the condition $K_X^2<\frac{9}{2}\chi(\O_X)$ is also necessary even if we are assuming that the Albanese morphism is composed with an involution.
Otherwise a double cover $X$ of $Y=A\times B$ with $g(A)=g(B)=2$ ramified over a divisor with at most negligible singularities linearly equivalent to the canonical divisor $K_Y$ of $Y$ gives a counterexample.
\end{remark}

\medskip

The quotient surface $Y=X/i$ can be  singular, but its singular points are not so bad: they are $A_1$ singularities and they are in one-to-one correspondence with the isolated fixed points of $i$. Let $Y'$ be the resolution obtained by blowing up the singularities and let $X'$ be the blow-up of $X$ at the isolated fixed points of $i$. Denote by $Y_0$ the minimal model of $Y'$ and by $X_0$ the middle term of the Stein factorization of the morphism from $X'$ to $Y_0$. What we get is the following commutative diagram.
\begin{equation}
\label{eq_ancora}
\begin{tikzcd}
X \arrow{d}{\pi} & X' \arrow{d}{\pi'}\arrow{l}[swap]{f_X}\arrow{r}{g_X} & X_0 \arrow{d}{\pi_0}\\
Y=X/i 					 & Y' \arrow{r}{g_Y} \arrow{l}[swap]{f_Y}								& Y_0.
\end{tikzcd}
\end{equation}
 We know that the double covers $\pi'$ and $\pi_0$ are  given by equations $2L'=\O_{Y'}(R')$ and $2L_0=\O_{Y_0}(R_0)$ respectively where $R'$ and $R_0$ are the branch divisors. Notice that $R_0$ has to be reduced (because $X_0$ is normal), while $R'$ has to be smooth (because $X'$ is smooth). It follows directly from the universal property of the Albanese morphism and the fact that $\alb_X$ factors through $\pi$ that $Y_0$ is a surface of maximal Albanese dimension with $q(Y_0)=q(X)$ which we call $q$ for simplicity.

By the classification of minimal surfaces, we know that $Y_0$ has non-negative Kodaira dimension and maximal Albanese dimension and in particular we have the following possibilities :
\begin{itemize}
	\item if $k(Y_0)=0$, then $Y_0$ is an Abelian surface and $q=2$;
	\item if $k(Y_0)=1$, then $Y_0$ is an  elliptic surface only with smooth (but possibly  multiple) fibres over a curve $C$ with genus $g(C)=q-1$ (cf. Lemma \ref{lemma_ueno});
	\item if $k(Y_0)=2$, then $Y_0$ is a minimal surface of general  type of maximal Albanese dimension with $q\geq 2$.
\end{itemize}

The surface $X_0$ may not be smooth, so we perform the canonical resolution. We get the following diagram
\begin{equation}
\label{eq_char0xtoy}
\begin{tikzcd}
X \arrow{d}{\pi} & X_t \arrow{d}{\pi_t}\arrow{l}[swap]{\phi}\arrow{r}{\phi_0} & X_0 \arrow{d}{\pi_0}\\
Y=X/i & Y_t \arrow{r}{\psi_0} \arrow{l}[swap]{\psi} & Y_0.
\end{tikzcd}
\end{equation}
 We notice that $X$ is nothing but the minimal model of $X_t$: thus, there exists an integer $n$ such that $\phi$ is the composition of $n$ blow-ups. 
 In particular $K_X^2=K_{X_t}^2+n$ and $\chi(\O_X)=\chi(\O_{X_t})$.

\begin{proposition}
\label{propo_k<2}
In the situation above, suppose that $Y_0$ is not a surface of general type.
Then $X$ satisfies
\[
K_X^2\geq 4\chi(\O_X)+4(q-2)
\]
and equality holds if and only if the canonical model of $X$ is isomorphic to a double cover of a product elliptic surface ($q\geq 3$) $Y=C\times E$ where $E$ is an elliptic curve and $C$ is a curve of genus $q-1$, whose branch divisor $R$ has a most negligible singularities and
\begin{equation*}
R\sim_{lin} C_1+C_2+\sum_{i=1}^{2d}E_i
\label{eq_teomioabranch}
\end{equation*}
where  $E_i$ (respectively $C_i$) is a fibre of the first projection (respectively the second projection) of $C\times E$ and $d>7(q-2)$ or the canonical model of $X$ is a double cover of an Abelian surfaces branched over an ample divisor ($q=2$). 
Moreover we have that $\Alb(X)=\Alb(Y)$ and $q(X)=q'(X)$.
\end{proposition}

\begin{proof}
If $Y_0$ is an elliptic surface we have (cf. Equation \ref{eq_canellsur})
\begin{equation}
\begin{split}
K_{Y_0}\sim_{num}(2g(C)-2-&\deg(\L_0))F+\sum_{j=1}^m(a_j-1)F_j=\\
= 2(q-2)F-\deg(\L_0)F+&\sum_{j=1}^m(a_j-1)F_j
\end{split}
\label{eq_candegell234}
\end{equation}
(recall that $\L_0$ is the locally-free part of $R^1(\alpha_0)_*\O_{Y_0}$ where $\alpha_0\colon Y_0\to C_0$ is the elliptic fibration, $F$ is a general fibre and $F_i$ are the multiple fibres of $\alpha_0$  and that $-\deg(\L_0)\geq0$, cf. Equation \ref{eq_degl0}).
Moreover, we have $F.L_0>0$ because $X$ is of general type. 
Indeed, assume by contradiction that $F.L_0\leq 0$.
In particular it follows that $\psi_0^*F.L_t=\psi^*F.(\psi_0^*L_0(-\sum_{i=1}^tm_iE_i))=F.L_0\leq 0$ which is a contradiction because $K_{X_t}=\pi_t^*(K_{Y_t}+L_t)$, $K_{Y_t}$ is,numerically, a multiple of a fibre and $X_t$ is of general type.
When $Y_0$ is Abelian, we know that the canonical bundle is trivial.

Using Equations \ref{cansq} and \ref{eulchar}, we get
\begin{equation}
\label{eq_ell1}
\begin{split}
K_X^2-4\chi(\O_X)=K_{X_t}^2&-4\chi(\O_{X_t})+n=\\
=2(K_{Y_0}^2-4\chi(\O_{Y_0}))+2K_{Y_0}&.L_0+2\sum_{i=1}^{t}(m_i-1)+n=\\
=4(q-2)F.L_0-&2\deg(\L_0)F.L_0+\\
+2\sum_{j=1}^m(a_j-1)F_j.L_0+&2\sum_{i=1}^{t}(m_i-1)+n\geq\\
\geq 4(q-2)+2\sum_{j=1}^m(a_j-1)F_j.L_0&+2\sum_{i=1}^{t}(m_i-1)+n\geq 4(q-2).
\end{split}
\end{equation}

It is then clear that equality holds if and only if $n=0$ (i.e. $X=X_t$), $m_i=1$ for all $i$,  $a_j=1$ for all $j$ and $\deg(\L_0)=0$ (the last two are required only when $Y_0$ is elliptic) and $K_{Y_0}.L_0=4(q-2)$.
The last equality is trivially true over an Abelian surface while over an elliptic surface is equivalent to $F.L_0=1$ which implies that there are no multiple fibres.
From this follows that there are no exceptional fibres (this is also implied by $\deg(\L_0)=0$, cf. Remark \ref{rem_irr}) and, by Proposition \ref{propo_isoell}, that $Y_0$ is a product elliptic surface (and if $q=2$ is a product Abelian surface), i.e. $Y_0=C\times E$ where $E$ is an elliptic curve and $C$ a curve of higher genus.
Observe that by \cite{conti} Lemma 5.3 we know that the branch divisor $R_0$ is linearly equivalent to $\Gamma_{g}+\Gamma_{g+e}+\sum_{i}E_i$ where $g$ is a non-constant function from $C$ to $E$, $e$ a closed point of $E$, $\Gamma(g)$ the graph of $g$  and $E_i$ different fibres of $C\times E\to C$.
Moreover (by ibidem Lemma 5.1) we know that $X_0$ is naturally isomorphic to a double cover of $C\times E$ branched over a divisor linearly equivalent to $C_1+C_2+\sum_{i}E_i$.
Notice that Lemma 5.1 and 5.3 of \cite{conti} are proved over the complex numbers, but actually the proofs work verbatim over any algebraically closed field of characteristic different from $2$.  

The fact that, in case of equality, we have $q'(X)=q'(Y)=q(Y)=q(X)$ follows easily from K\"unneth's formula and the explicit formula of the branch divisor when $Y_0$ is an elliptic surface, while, when $Y_0$ is Abelian it is even more immediate by the fact that $L_0$ has to be ample.
\end{proof}

Now we assume that in diagram \ref{eq_char0xtoy}, the surface $Y_0$ is of general type: in this case we have to prove that a stronger inequality holds for the invariants of $X$.

\begin{lemma}
In the situation above, we have:
\begin{enumerate}
	\item if $K_{Y_0}^2-4\chi(\O_{Y_0})\geq 4(q-2)$ then
	\[K_X^2-4\chi(\O_X)\geq 8(q-2)\geq 4(q-2)\]
	and equality holds if and only if $Y=Y_0$, $q=2$ and $\pi$ is an \'etale double cover;
	\item if $K_{Y_0}^2\geq\frac{9}{2}\chi(\O_{Y_0})$ then 
	\[K_X^2-4\chi(\O_X)>28(q-2).\]
\end{enumerate}
\label{lemma_luzuogen+}
\end{lemma}

\begin{proof}
Arguing similarly as in the proof of Proposition \ref{propo_k<2}, if $K_{Y_0}^2- 4\chi(\O_{Y_0})\geq 4(q-2)$ we have
	\begin{equation}
\label{eqgen1}
\begin{split}
K_X^2-4\chi(\O_X)&\geq 2(K_{Y_0}^2-4\chi(\O_{Y_0}))+2K_{Y_0}.L_0+2\sum_{i=1}^{t}(m_i-1)\geq\\
&\geq 8(q-2)\geq 4(q-2).
\end{split}
\end{equation}
Equality would be possible if and only if  $q=2$, $K_{Y_0}^2=4\chi(\O_{Y_0})$, $X=X_t$ is minimal, $m_i=1$ for all $i$ and $K_{Y_0}.R_0=0$.
In particular the last inequality implies that $R_0$ is a, possibly empty, union of $(-2)$-curves (these are the only curves with non-positive intersection with the canonical bundle).
Suppose that $R_0$ is not trivial: then we would have
\begin{equation}
(K_{Y_t}+\frac{1}{2}R_t).\psi_0^*R_0=\Bigl(\psi_0^*(K_{Y_0}+\frac{1}{2}R_0)+\sum_{i=1}^t(1-m_i)E_i\Bigr).\psi_0^*R_0=\frac{1}{2}R_0^2< 0
\label{eq_nonnefcan}
\end{equation}
which would contradict the nefness of the canonical bundle of the minimal surfaces $X=X_t$ (by Equation \ref{eq_candoubcov}, $K_{X_t}=\pi_t^*(K_{Y_t}+\frac{1}{2}R_t)$ and $\pi_t^*\psi_0^*R_0$ is clearly effective).
By this we have proved $X_0=X_t=X$, $Y_0=Y_t=Y$ and $\pi\colon X\to Y$ is an \'etale double cover such that $\alb_X=\alb_Y\circ\pi$ and the first part of the Lemma is proven.

 If $K_{Y_0}^2\geq\frac{9}{2}\chi(\O_{Y_0})$ then, thanks to Equations \ref{cansq} and \ref{eulchar},  we have that 
\begin{equation}
\begin{split}
0>K_X^2-\frac{9}{2}\chi(\O_X)\geq& K_{X_t}^2-\frac{9}{2}\chi(\O_{X_t})=\\
=2\Bigl(K_{Y_0}^2-\frac{9}{2}\chi(\O_{Y_0})\Bigr)+\frac{7}{4}K_{Y_0}L_0-&\frac{1}{4}L_0^2+\frac{1}{4}\sum_{i=1}^t(m_i-1)(m_i+8)\geq\\
\geq \frac{1}{4}\Bigl(7K_{Y_0}.&L_0-L_0^2\Bigr).
\end{split}
\label{eq_>2gentype}
\end{equation}
In particular we have $L_0^2>7K_{Y_0}.L_0\geq 0$ where the last inequality follows from the fact that $2L_0=R_0$ is effective and $Y_0$ is minimal.

By the Hodge index Theorem we have that $(K_{Y_0}.L_0)^2\geq K_{Y_0}^2L_0^2$ which, combined with Equation \ref{eq_>2gentype}, guarantees that
\begin{equation}
L_0^2 (K_{Y_0}.L_0)>7(K_{Y_0}.L_0)^2\geq 7K_{Y_0}^2L_0^2
\label{eq_>2genhodge}
\end{equation}
i.e.
\begin{equation}
K_{Y_0}.L_0>7K_{Y_0}^2.
\label{eq_>2genhodge2}
\end{equation}

Therefore we can rephrase  Equation \ref{eqgen1} as
\begin{equation}
\begin{split}
K_X^2-4\chi(\O_X)\geq 2(K_{Y_0}^2&-4\chi(\O_{Y_0}))+2K_{Y_0}.L_0>\\
>\chi(O_{Y_0})+14K_{Y_0}^2> 28(p_g(Y_0&)-2)\geq 28(q'(Y_0)-2)\geq\\
\geq 28(q(Y_0)-2)&=28(q(X)-2)
\end{split}
\label{eq_y0gentype}
\end{equation}
where we are using Noether's inequality $K_{Y_0}^2\geq 2p_g(Y_0)-4$ (cf. \cite{liedtkec2} Theorem 2.1) and $\chi(\O_{Y_0})>0$ for a minimal surface of general type (cf. \cite{gu2019slope} Corollary 3.4).
\end{proof}

\begin{remark}
\label{rem_8q-2}
Observe that, in a completely similar way as we have done in the first part of Lemma \ref{lemma_luzuogen+}, we prove that 
\begin{equation}
K_X^2-4\chi(\O_X)=8(q-2)
\label{eq_nuova8q-2}
\end{equation}
if and only if $X=X_t=X_0$, $Y=Y_t=Y_0$ and $\pi\colon X\to Y$ is an \'etale cover and the invariants of $Y$ satisfy $K_Y^2-4\chi(\O_Y)=4(q-2)$.
We will need this later when we study surfaces close to the Severi lines.
\end{remark}

\begin{proof}[Proof of Theorem \ref{teo_mio3}]
Thanks to Proposition \ref{propo_k<2} and Lemma \ref{lemma_luzuogen+} we just have to analyse further the case when $Y_0$ is of general type.

If $K_{Y_0}^2\geq\frac{9}{2}\chi(\O_{Y_0})$ then we can conclude by part two of Lemma \ref{lemma_luzuogen+}.
On the other hand, if we assume that $K_{Y_0}<\frac{9}{2}\chi(\O_{Y_0})$, thanks to Theorem \ref{teo_mio1}, we know that there exists an involution on $Y_0$ relative to the Albanese morphism $\Alb_{Y_0}$ of $Y_0$.
Proceeding by induction on the power of $2$ dividing the degree of $\alb_X$, we get a sequence of minimal surfaces $Y_i$ with $i=1,\dots,n$ with rational maps of degree two $\phi_i\colon Y_i\dashrightarrow Y_{i+1}$ such that the Albanese morphism $\alb_X$ of $X$ factors through the $\phi_i$, $Y_i$ is a surface of general type whose invariants satisfy $K_{Y_i}^2<\frac{9}{2}\chi(\O_{Y_i})$ for $i=1,\ldots,n-1$ and $Y_n$ is either a surface of general type with $K_{Y_n}^2\geq\frac{9}{2}\chi(\O_{Y_n})$ or a surface not of general type.
In both cases, through the above discussion, we have that $K_{Y_{n-1}}^2\geq 4\chi(\O_{Y_{n-1}})+4(q-2)$ and proceeding by induction and using Equation \ref{eqgen1} we conclude that $K_X^2\geq 4\chi(\O_X)+4(q-2)$.

The fact that if $Y_0$ is of general type we can not have equality, proceeds as follows.
Indeed, suppose that $K_X^2-4\chi(\O_X)=4(q-2)$, we have that $Y=Y_0$, $\pi\colon X\to Y$ is an \'etale double cover, $K_Y^2=4\chi(\O_Y)$ and $q=2$ thanks to part one of Lemma \ref{lemma_luzuogen+}.
Proceeding by induction on the power of $2$ dividing the degree of $\alb_X$, we obtain, similarly as above, a sequence
\begin{equation}
X=Z_0\xrightarrow{\phi_1=\pi}Y=Z_1\xrightarrow{\phi_2}Z_2\xrightarrow{\phi_3}\ldots\xrightarrow{\phi_{n-1}}Z_{n-1}\xrightarrow{\phi_n}Z_n=\Alb(X)
\label{eq_mavero}
\end{equation}
where $n\geq2$, $Z_i$ is a surface of general type for $i=0\ldots,n-1$, $\phi_i$ is an \'etale double cover for $i=1,\ldots,n-1$, $Z_n$ is the Albanese variety of $X$ and $\phi_n$ is the resolution of singularities of a flat double cover branched over an ample divisor.
By Theorem \ref{teo_funddoub}, we know that $\phi_{n}$ induces an isomorphism of algebraic fundamental groups. 
In particular there exists an \'etale double cover $A\to \Alb(X)$ where $A$ is an Abelian surface such that 
\begin{equation}
\begin{tikzcd}
Z_{n-2}\arrow{r}{\phi_{n-1}}\arrow{d}{}&Z_{n-1}\arrow{d}{\phi_n}\\
A\arrow{r}{}& \Alb(X)
\end{tikzcd}
\label{eq_mahvero}
\end{equation}
is a Cartesian diagram.
But this is a contradiction to the fact that $\phi_n\circ\ldots\circ\phi_1=\alb_X\colon X\to \Alb(X)$.
The proof  is then complete.
\end{proof}

\section{Surfaces close to the Severi lines in characteristic different from $2$}
\label{sec_close}

In this section we prove Theorem \ref{teo_mio4}, i.e. we give a characterization of surfaces which lie close to the Severi lines over an algebraically closed field of characteristic different from $2$ and we will see that the closest satisfy $K_X^2=4\chi(\O_X)+8(q-2)$ (we are always assuming $K_X^2<\frac{9}{2}\chi(\O_X)$).

First of all we prove the following Lemma, which we need to show that in case $K_X^2=4\chi(\O_X)+8(q-2)$ holds we have that the Picard scheme is reduced, i.e. $q'=q$.
\begin{lemma}
Let $\alpha\colon Y\to C$ be a smooth elliptic surface fibration with maximal Albanese dimension and let $X$ be the minimal smooth model of the double cover given by the equation $2L=\O_Y(R)$ where $R$ has at most  negligible singularities, denote by $\pi\colon X\to Y$ the induced morphism.
Suppose moreover that $X$ is a surface of general type, $q(X)=q(Y)$, $L.F=2$ where $F$ is a fibre of $Y\to C$ and $K_X^2<\frac{9}{2}\chi(\O_X)$.
Then we have that $q'(X)=q(X)$
\label{lemma_q'=q}
\end{lemma}

\begin{proof}
By %Corollary \ref{coro_ellsurftrivet}
Remark \ref{rem_irr}, we know that $q(Y)=q'(Y)=g(C)+1$ and by Proposition \ref{propo_isoell} we know that $Y\to C$ becomes trivial after an \'etale cover $\delta\colon\widetilde{C}\to C$ of degree $2$ (here we are using $L.F=2$ and $\cha(k)\neq 2$).
In particular we have the following Cartesian diagram
\begin{equation}
\begin{tikzcd}
\widetilde{X}\arrow{r}{\beta}\arrow{d}{\widetilde{\pi}}& X\arrow{d}{\pi}\\
\widetilde{Y}=\widetilde{C}\times F\arrow{d}{\widetilde{\alpha}}\arrow{r}{\gamma} & Y\arrow{d}{\alpha}\\
\widetilde{C}\arrow{r}{\delta} & C.
\end{tikzcd}
\label{eq_hofinitoleidee}
\end{equation} 

If we assume that $L$ is ample we conclude by
\begin{equation}
\begin{split}
q'(X)=h^1(X,\O_X)&=h^1(Y,\pi_*\O_X)=\\
=h^1(Y,\O_{Y})+h^1(Y,L^{-1})=&q'(Y)=q(Y)=q(X)
\end{split}
\label{eq_q'=q}
\end{equation}
where $h^1(Y,L^{-1})=0$ follows by Lemma \ref{lemma_kodvan}.

Notice that $L$ is ample if and only if $\widetilde{L}=\gamma^*L$ is, which is the line bundle associated with the double cover $\widetilde{\pi}$.
We know that $\widetilde{L}$ is effective up to a multiple (twice $\widetilde{L}$ is the ramification divisor of $\widetilde{\pi}$) and, by the formula of the Picard group for $\widetilde{C}\times F$ (cf. Proposition \ref{propo_exactpicce}), we know it is numerically equivalent to $\Gamma_f+\widetilde{C}+lF$ where $l\geq0$, $F$ and $\widetilde{C}$ are fibres of $\widetilde{\alpha}$ and $\widetilde{C}\times F\to F$ respectively and $f\colon\widetilde{C}\to F$ is a morphism.
If $f$ is not a constant function, it is clear by the Nakai-Moishezon criterion that $\widetilde{L}$ is an ample line bundle (recall that $\Gamma_f^2=0$ because we can see it as a fibre of a morphism $\widetilde{C}\times E\to E$, cf. \cite{conti} ).
So suppose that $f$ is constant, i.e. $\widetilde{L}$ is numerically equivalent to $2\widetilde{C}+lF$. 
If $l=0$ we would have that $\widetilde{L}^2=0$ and $\widetilde{L}.K_{\widetilde{Y}}>0$.
In particular Equations \ref{cansq} and \ref{eulchar} and the fact that for \'etale covers the Euler characteristic and the self intersection of the canonical divisor are multiplicative tell us that
\begin{equation}
\begin{split}
0>K_X^2-\frac{9}{2}\chi(\O_X)=\frac{1}{d}(K_{\widetilde{X}}^2-\frac{9}{2}\chi(\O_{\widetilde{X}}))=\\
=\frac{2}{d}(K_{\widetilde{Y}}^2-\frac{9}{2}\chi(\O_{\widetilde{Y}}))+\frac{7}{4d}\widetilde{L}.K_{\widetilde{Y}}-\frac{1}{4d}\widetilde{L}^2>0
\end{split}
\label{eq_bohultcha+}
\end{equation}
which is a contradiction.
So $l>0$ and, again by the Nakai-Moishezon criterion, $\widetilde{L}$ is ample also in this case.
\end{proof}

We are now ready to prove Theorem \ref{teo_mio4}.

\begin{proof}[Proof of Theorem \ref{teo_mio4}]
We use the same notation as in Section \ref{sec_severi}.

If $Y_0$ is an Abelian surface the first part of the Theorem is trivial.

If $Y_0$ is a surface of general type the first part follows from Lemma \ref{lemma_luzuogen+} and Theorem \ref{teo_mio3}.

If $Y_0$ is an elliptic surface, by Equation \ref{eq_ell1} we have that 
\begin{equation}
\begin{split}
K_X^2-4\chi(\O_X)=4(q-2)F&.L_0-2\deg(\L_0)F.L_0+\\
+2\sum_{j=1}^m(a_j-1)F_j.L_0+&2\sum_{i=1}^{t}(m_i-1)+n.
\end{split}
\label{eq_ell1b}
\end{equation}
As already observed in Proposition \ref{propo_k<2}, $K_X^2-4\chi(\O_X)=4(q-2)$ holds if and only if $F.L_0=1$, $\deg(\L_0)=0$, $a_j=1$ for all $j$, $m_i=1$ for all $i$ and $n=0$.

If we want to increase slightly $K_X^2-4\chi(\O_X)$, we thus have $5$ possibilities.

First we discuss $n$. We know that if all the $m_i=1$, then all the irreducible components of the exceptional curve in the cover surface are $(-2)$-curves (cf. \cite{barth} table 1 page 109). Moreover these are the only possible rational curves on $X_t$ (ibidem). This means that in this case $n=0$. In particular, if $n>0$, then there exists an $i$ such that $m_i>1$.

Now suppose that there exists an $i$ such that $m_i>1$. By the classification of simple singularities of curves (cf. \cite{barth} II.8) we know that we have two possibilities for $R_0$. If $R_0$ has a singular point $x$ of order greater than or equal to $4$, then $(F.R_0)_x\geq 3$ (it may happen that one of the irreducible components of $R_0$ passing through $x$ is a fibre). Hence $F.R_0\geq 4$ because $R_0=2L_0$. The other possibility is that $R_0$ has a triple point $x$ which is not simple. A necessary condition for a  triple point not to be simple is to have a single tangent. If the tangent of $R_0$ in $x$ is transversal to $F$, then $(F.R_0)_x\geq 3$, conversely if it is tangent to $F$, we have $(F.R_0)_x\geq 4$ (it may happen, as before, that one of the irreducible component is $F$ itself). In both cases we have $F.R_0\geq 4$.

Notice that if $\deg(\L_0)<0$ or there exists an index $j$ such that $a_j>1$ then there exists at least a multiple fibre in the elliptic surface $Y_0$.

Suppose now that $Y_0$ has a multiple fibre $F_j$ with multiplicity $n_j$. In this case we have $F.R_0=n_jF_j.2L_0\geq 2n_j\geq 4$.

To summarize, whatever quantity we increase, we get $F.R_0\geq 4$: that is to say that whenever 
\[K^2_X>4\chi(\O_X)+4(q-2),\]
we get
\[K^2_X\geq 4\chi(\O_X)+8(q-2)\]
and part 1 is proven.

Now we study the case when $q=2$.
If $Y_0$ is a surface of general type and satisfies $K_{Y_0}^2\geq\frac{9}{2}\chi(\O_{Y_0})$ we can easily conclude using to part 2 of Lemma \ref{lemma_luzuogen+}.
Thanks to Theorem \ref{teo_mio3} we have that in all other cases $K_{Y_0}^2-4\chi(\O_{Y_0})\geq 0$ holds (if $Y_0$ is not of general type, we actually have equality) so that we get
\begin{equation}
\begin{split}
K_X^2-4\chi(\O_X)=2(K_{Y_0}^2-4\chi(\O_{Y_0}))+K_{Y_0}.R_0+2\sum_{i=1}^{t}(m_i-1)+n=\\
=2\Bigl(K_{Y_0}^2-4\chi(\O_{Y_0})+K_{Y_0}.L_0+\sum_{i=1}^{t}(m_i-1)\Bigr)+n:
\end{split}
\end{equation}
as above, if $n>0$, there exists an $i$ such that $m_i>1$. 
In particular $K_X^2>4\chi(\O_X)$ implies $K_X^2\geq 4\chi(\O_X)+2$ and part two of the Theorem is proven.

Now suppose that equality holds, i.e. $K_X^2=4\chi(\O_X)+8(q-2)<\frac{9}{2}\chi(\O_X)$ with $q\geq 3$: it is enough to prove that $Y_0$ is an elliptic surface. Indeed, in this case, it is immediate from Equation \ref{eq_ell1b} and the discussion thereafter that the conditions of the Theorem are necessary and sufficient. 

Suppose by contradiction that $Y_0$ is a surface of general type. For the numerical invariants of $Y_0$, thanks to Theorem \ref{teo_mio3} and Lemma \ref{lemma_luzuogen+}, we have two possibilities. First, if $K_{Y_0}^2\geq \frac{9}{2}\chi(\O_{Y_0})$, we know that $K_X^2-4\chi(\O_{X})> 28(q-2)>8(q-2)$ which is a contradiction. The other possible case is if
\begin{equation}
\label{eq_ultimaaggiunta}
\frac{9}{2}\chi(\O_{Y_0})>K_{Y_0}^2\geq 4\chi(\O_{Y_0})+4(q-2)
\end{equation}
and, from Remark \ref{rem_8q-2}, we derive that there $Y_0=Y$, its invariants satisfy
\begin{equation}
\frac{9}{2}\chi(\O_{Y})>K_{Y}^2= 4\chi(\O_{Y})+4(q-2)
\label{eq_ultimaaggiuntabis}
\end{equation} and $\pi$ is an \'etale double cover which is a factor of the Albanese morphism $\Alb_X$ of $X$.
The fact that the invariants of $Y$ satisfy Equation \ref{eq_ultimaaggiuntabis}, implies that (thanks to Theorem \ref{teo_mio3}) the canonical model of $Y$ is a double cover of a product elliptic surface $C\times E$ branched over an ample divisor $R$ and we call $\alpha\colon Y\to C\times E$ the induced generically finite morphism.
By Theorem \ref{teo_funddoub}, we know that $\phi$ induces an isomorphism on the algebraic fundamental groups; in particular there exists an \'etale double cover $\beta\colon \overline{Y}\to C\times E$ and a degree-two morphism $\gamma\colon X\to \overline{Y}$ branched over an ample divisor such that
\begin{equation}
\begin{tikzcd}
X\arrow{r}{\pi}\arrow{d}{\gamma}& Y\arrow{d}{\alpha}\\
\overline{Y}\arrow{r}{\beta} & C\times E
\end{tikzcd}
\label{eq_nuovaagg}
\end{equation}
is a Cartesian diagram.
If we prove that the Albanese morphism $\Alb_{\overline{Y}}$ of $\overline{Y}$ has degree $1$ onto its image we get a contradiction: indeed in this case, because $\pi$ is a factor of the Albanese morphism $\Alb_X$ of $X$ and the morphism $X\to\Alb_{\overline{Y}}$ has degree two onto its image, we obtain that the $\Alb_X$ has degree two onto its image and, in particular, that $\overline{Y}=Y$ which is not possible because, for example, $Y$ is of general type while $\overline{Y}$ is an elliptic surface.
We know that  $\overline{Y}$ as an \'etale double cover of $C\times E$ is uniquely determined by an element $L\in\Pic(C\times E)$ of order $2$. 
By Proposition \ref{propo_exactpicce} it follows at once that $L\in\Pic^0(C\times E)$: indeed there are  no torsion elements  in $\Hom_{c_0}(C,E)$ and the torsion line bundles of a curve have degree zero.
From this we can derive that there exists an \'etale double cover $A\to\jac(C)\times E$ where $A$ is clearly an Abelian variety and a morphism $\overline{Y}\to A$ such that
\begin{equation}
\begin{tikzcd}
\overline{Y}\arrow{r}{\beta}\arrow{d}{}& C\times E\arrow{d}{}\\
A\arrow{r}{} & \jac(C)\times E
\end{tikzcd}
\label{eq_nonsoh}
\end{equation}
is a Cartesian square.
The morphism $\overline{Y}\to A$ is a closed inclusion because $C\times E\to\jac(C)\times E$ is, hence the Albanese morphism $\Alb_{\overline{Y}}$ of $\overline{Y}$ has degree $1$ and the proof is complete.

\end{proof}

\section{Partial results in characteristic $2$}
\label{sec_char2}
In this section we are going to see what can be said about our results when the characteristic of the ground field is $2$.
As we have already mentioned, the results presented on Sections \ref{sec_severi} and \ref{sec_close} are based on the theory of double covers of surfaces and that's why things get wilder when the characteristic of the ground field is exactly $2$.
In Section \ref{sec_pre} we have shown the theory of double covers, with a particular attention to the case of characteristic $2$: even though this case is more complicated we have some similarities.

First of all we noticed that for every double cover $f\colon X\to Y$ ther is naturally associated a line bundle $L$ via the following short exact sequence:
\begin{equation}
0\to \O_Y\to f_*\O_X\to L^{-1}\to 0.
\label{eq:}
\end{equation}

If the characteristic is different from two, we have seen that this exact sequence splits and the divisor associated with $L^2$ is the ramification divisor of $f$: in particular we have that $L$ is an effective divisor up to a multiple.
We have used this effectiveness result for $L$  to show that $K_Y.L\geq 0$ when $Y$ is a minimal non-ruled surface, and this is the real result we would like to extend in characteristic two.

As we have already noticed in Remark \ref{rem_effdouble}, the line bundle $L$ is effective up to a multiple in most cases even when the characteristic of the ground field is $2$.
There is just one single case where nothing assures $L$ to be effective up to a multiple, i.e. for non-splittable inseparable double covers.

Now suppose that $f\colon X\to Y$ is an inseparable double cover such that $f$ is a factor of the Albanese morphism $\alb_X\colon X\to \Alb(X)$. 
For simplicity we also assume that both $X$ and $Y$ are smooth.
As we have seen, it is canonically defined another double cover $g^{(1)}\colon Y^{(1)}\to X$. %where $Y^{(1)}$ is the pull-back of $Y$ via the absolute Frobenius morphism.

As we have seen in Remark \ref{rem_albinsfact}, the line bundle $M^{-1}$, which is the inverse of the line bundle associated with $g^{(1)}\colon Y^{(1)}\to X$, has many sections.
At first glance, we thought it were possible to prove that $L$ has to be effective up to a multiple using this and the relations of Lemma \ref{lemma_flgm} between $L$, $M$ and the associated foliations on $Y^{(1)}$ and $X$, but in the end we were not able to do so.

Without this, we can only collect here partial results concerning Severi type inequalities when the characteristic of the ground field is $2$.

Let $X$ be a surface of general type with maximal Albanese dimension over a field of characteristic $2$ satisfying $K_X^2<\frac{9}{2}\chi(\O_X)$: by Theorem \ref{teo_mio1} we know that there exists a morphism of degree two $\pi\colon X\to Y$ with $Y$ normal such that
\[
\begin{tikzcd}
X \arrow{rr}{\pi}\arrow[swap]{dr}{\alb_X}& & Y\arrow{dl}{} \\
 & \Alb(X) &
\end{tikzcd}
\]
commutes.

As we have done over  fields of characteristic different from two (cf. Equation \ref{eq_ancora}) we obtain the following diagram
\begin{equation}
\begin{tikzcd}
X\arrow{d}{\pi}& X_t\arrow{l}{\phi}\arrow{d}{\pi_t}\arrow{r}{\phi_0} & X_0\arrow{d}{\pi_0}\\
Y & Y_t\arrow{l}{\psi}\arrow{r}{\psi_0} & Y_0
\end{tikzcd}
\label{eq_diagpos2}
\end{equation}
where $Y_0$ is the minimal smooth model of $Y$ and $\pi_t\colon X_t\to Y_t$ is the canonical resolution of $\pi_0\colon X_0\to Y_0$.

\begin{lemma}
\label{lemma_specialtype+}
Suppose that $Y_0$ has Kodaira dimension strictly smaller than $2$.
Then we have that 
\begin{equation}
K_X^2\geq 4\chi(\O_X)+4(q-2)
\label{eq_+sevegen}
\end{equation}
and equality holds if and only if the canonical model of $X$ is isomorphic to a double cover of a product elliptic surface $C\times E$ ($q\geq 3$) or of an Abelian surface ($q=2$).

Suppose moreover that $\pi_0$ is separable and $q\geq3$: in this case we can further prove that the line bundle $L_0$ associated with $\pi_0$ is effective and is linearly equivalent to 
\begin{equation}
L\sim_{lin}C+\sum_{i=1}^dE_i
\label{eq_lineqchar2}
\end{equation}
where, as usual, $C$ (respectively $E_i$) is a fibre of the first projection (respectively are fibres of  the second projection) of $C\times E$  and $d>7(q-2)$ and $q'(X)=q(X)$, i.e. the Picard scheme of $X$ is reduced.
\end{lemma}

\begin{proof}
If $Y_0$ is an Abelian surface it is clear that everything works as over a field of characteristic different from two.

If $Y_0$ is an elliptic surface we have seen in the proof of Proposition \ref{propo_k<2}, we have that $F.L_0>0$ and the proof there given only relies on the fact that $X$ is a surface of general type and the canonical bundle of $Y_0$ is numerically equivalent to a rational multiple of $F$, so it works even if we do not know whether $L_0$ is effective up to a multiple or not.
Once this has been noticed, the Lemma is proven by Equation \ref{eq_ell1} and Proposition \ref{propo_isoell}.

If $\pi_0$ is separable, the linear class of $L$ is proved in a similar, and simpler, way as in Lemma 5.3 of \cite{conti}.
The fact that $d>7(q-2)$ is required in order to have $K_X^2<\frac{9}{2}\chi(\O_X)$ (cf. Example \ref{es_1}) which is an assumption for all this Section.
Then K\"unneth formula and the long exact sequence in cohomology associated with the short exact sequence associated with the double covers $\pi_0\colon X_0\to Y_0$ shows that $q'(X)=g(E)+g(C)=q(X)$.
\end{proof}

\begin{definition}
Let $X$ be a surface of general type of maximal Albanese dimension.
Suppose that the Albanese morphism $\alb_X$ of $X$ factorizes as
\begin{equation}
X\xrightarrow{f_1} X_1\xrightarrow{f_2}\ldots \xrightarrow{f_n} X_n \xrightarrow{g} \Alb(X)
\label{eq_star}
\end{equation}
where the minimal models $\widetilde{X}_1, \ldots, \widetilde{X}_{n-1}$ of $X_1, \ldots, X_{n-1}$ are surfaces of general type satisfying $K_{\widetilde{X}_i}<\frac{9}{2}\chi(\O_{\widetilde{X}_i})$, $f_i$ are finite morphism of degree $2$ and either $\widetilde{X}_n$ is not of general type or  $\widetilde{X}_n$ is of general type and $K_{\widetilde{X}_n}\geq\frac{9}{2}\chi(\O_{\widetilde{X}_n})$.
In this situation, we say that $X$ has the property (*) if it satisfies the following condition: the line bundle $\widetilde{L}_i$ associated with the double cover $\widetilde{f}_i\colon\widetilde{X}_{i-1}\to\widetilde{X}_i$ intersects non negatively the canonical bundle of $\widetilde{X}_i$ for every $i$ such that the minimal model of $X_i$ is of general type.
\label{def_star}
\end{definition}

\begin{remark}
Observe that the condition (*) is very technical, hence we give some stronger conditions which assure that it holds. 
If all the $f_i$ defined above are separable or splittable and inseparable (e.g. $\alb_X$ is a separable morphism), we clearly have that $X$ satisfies (*): indeed in this case we have seen that $\widetilde{L}_i$ is effective up to a multiple (cf. Remark \ref{rem_effdouble}).
\label{rem_star}
\end{remark}

\begin{theorem}
Let $X$ be a surface of general type with maximal Albanese dimension satisfying $K_X^2<\frac{9}{2}\chi(\O_X)$ over a field of characteristic $2$ for which condition (*) holds. 
Then we have $K_X^2-4\chi(\O_X)\geq 4(q-2)$ and equality holds if and only if it is a flat double cover of the blow-up of a product elliptic surface $C\times E$ ($q\geq 3$) or of an Abelian surface ($q=2$).
\label{teo_star2}
\end{theorem}

\begin{proof}
The condition (*) is exactly the one that assures that the same proof over a field of characteristic different from two works also in this setting.
\end{proof}

\bibliographystyle{plain} %stile che mi serve per far compare l'url della tesi di dottorato
\bibliography{bibliografia}

\end{document}